 \documentclass[reqno,11pt,a4paper]{amsart}

\usepackage{amsfonts,amsmath,amssymb}
\usepackage[latin1]{inputenc}
\usepackage{dsfont}
\usepackage{enumerate}
\usepackage{xcolor}
\usepackage[all]{xy}
\def\notshow#1\notshowend{} %

\newcommand{\C}{\mathcal{C}}

\newcommand{\s}{S_0}

\newcommand{\df}{\mathrm{d}}
\usepackage{amssymb}
\usepackage{amsfonts}
\usepackage{mathrsfs}
\usepackage{hyperref}

\usepackage[normalem]{ulem} 

\def\bb#1\eb{\textcolor{blue}{#1}} 
\def\br#1\er{\textcolor{red}{#1}} %

\newcommand{\R}{\mathds R}
\newcommand{\N}{\mathds N}

\newcommand{\n}{{n_0}}
\newcommand{\m}{{n_1}}

\newcommand{\B}{{B_0}}

\newcommand{\av}{{\mathfrak{a}}}
	\newcommand{\bv}{{\mathfrak{b}}}

\newcommand{\F}{{\tilde{F}}}
\newcommand{\hF}{{F}}
\newcommand{\tF}{\tilde{F}}
\newcommand{\g}{g_R} 
\newcommand{\cF}{{\hat{F}}}
\newcommand{\hess}{\hbox{Hess}\, }

\newtheorem{thm}{Theorem}[section]
\newtheorem{prop}[thm]{Proposition}
\newtheorem{lemma}[thm]{Lemma}
\newtheorem{cor}[thm]{Corollary}
\theoremstyle{definition}
\newtheorem{defi}[thm]{Definition}

\newtheorem{exe}[thm]{Example}

\newtheorem{rem}[thm]{Remark}

\title[Cones and Finsler spacetimes]{On the definition and examples of \\ cones and Finsler spacetimes  
}

\author[M. A. Javaloyes]{Miguel Angel Javaloyes}
\address{Departamento de 
Matem\'aticas, \hfill\break\indent
Universidad de Murcia, \hfill\break\indent
Campus de Espinardo,\hfill\break\indent
30100 Espinardo, Murcia, Spain}
\email{majava@um.es}

\author[M. S\'anchez]{Miguel S\'anchez}
\address{Departamento de Geometr\'{\i}a y Topolog\'{\i}a, Facultad de Ciencias, \hfill\break\indent
 Universidad de Granada,\hfill\break\indent
 Campus Fuentenueva s/n,
 \hfill\break\indent 18071 Granada, Spain}
\email{sanchezm@ugr.es}

\begin{document}
\begin{abstract}
 A systematic study of (smooth, strong) cone structures $\C$  and   Lorentz-Finsler  metrics $L$ is carried out.  As a  link between  both notions, cone triples $(\Omega,T, F)$, where $\Omega$ (resp. $T$) is a 1-form (resp. vector field) with $\Omega(T)\equiv 1$ and $F$, a Finsler metric on $\ker (\Omega)$,  are introduced.  
Explicit descriptions of all the Finsler spacetimes are given, paying special attention to stationary and static ones, as well as to issues related to differentiability.

 In particular,  cone structures $\C$ are bijectively associated  with  classes of anisotropically conformal metrics $L$, and the notion of {\em cone geodesic} is introduced consistently with both structures.     
As a non-relativistic application, the {\em time-dependent} Zermelo navigation problem is posed rigorously, and its general solution is provided. 

\vspace{10mm}

\end{abstract}
\maketitle

\tableofcontents

\section{Introduction}

The definition of Finsler spacetimes has been somewhat uncertain from the very beginning.  There are several issues that make it difficult: 
\begin{enumerate}[(i)]
\item  the generality inherent to  Finsler metrics, as one has a different (non-definite) scalar product for every direction in every tangent space,  
\item  the possible non-reversibility of the metric makes the distinction   between future and past harder,
\item there are many examples with smoothness issues, or having Lorentzian index only in some directions,
\item   there are Lorentz-Finsler elements in  physical models (such as birefringence crystals) which, in principle, might be independent of cosmological interpretations but they might be included in a mathematical notion of Finslerian spacetime, see \cite[\S V(B)]{PW11}. 
\end{enumerate}
As a byproduct of such difficulties, there are many different definitions of Finsler spacetimes spread in literature \cite{AaJa16,Beem70,CaStan16,CaStan18,Kos11,LPH12,PBPSS17,PW11,Vacaru} as discussed in Appendix \ref{s_a1}.
In this paper, we will try to clarify these issues,  developing systematically  a simple  definition of Finsler spacetimes with  good mathematical properties  (which will be applicable to other definitions), providing different ways to construct them and characterizing all  possible examples\footnote{Physical motivations for this definition were discussed in the meeting on Lorentz-Finsler Geometry and Applications '19
\url{http://gigda.ugr.es/finslermeeting/} 
and will 
be the aim of a future work. }. 

  From a mathematical viewpoint,  following Beem's approach \cite{Beem70}, it is natural to consider a Lorentz-Finsler metric as a (two-homogeneous) pseudo-Finsler one $L$ with fundamental tensor of coindex 1 defined on all the tangent bundle. However, the possible existence of more than two cones at each tangent space or the natural non-reversibility of Finsler metrics may obscure the physical intuition of  Finsler spacetimes as  extensions of  (classical) relativistic spacetimes. 
   These issues underlie  the multiplicity of proposed alternatives. However, notice that  it is not  clear the role of $L$ away from the  future causal cone\footnote{Notice that, as emphasized by Ishikawa \cite[formulas (3.2), (3.3)]{Ishi81}, 
one should {\em not} use the fundamental tensor $g_l$ for $l$ spacelike even when one measures a spacelike separation. Instead,  $g_v(l,l)$ for a  lightlike (or eventually timelike) vector $v$ should be used.},  and its seems natural  to maintain the Lorentzian signature at least  in that cone.   So, we will focus on Lorentz-Finsler metrics defined on a cone domain, which will include just the (future) timelike and causal  directions.   Therefore, a first task will be to clarify the relation between Lorentz-Finsler metrics and {\em cone structures}, being the  latter  smooth distributions of {\em strong} cones. This is carried out by means of a self-contained development along the paper, which can be summarized as follows.

\begin{thm}[Equivalence cone structures/ classes of Lorentz-Finsler metrics] Let $\C$ be a cone structure on $M$ (defined as a hypersurface of $TM$, according to Def. \ref{d_conestructure}). Then:

\begin{enumerate}
\item  $\C$  yields a natural notion of causality and, then, of {\em cone geodesics} (i.e., locally horismotic curves, Def. \ref{def_cone_g}).
 
\item $\C$ becomes equivalent to a class of anisotropically conformal Lorentz-Finsler metrics $L$, in the following sense:

\begin{enumerate}
\item Any Lorentz-Finsler metric $L$ on $M$ (Def. \ref{finslerst}) is endowed with a natural cone structure $\C$  (Cor. \ref{c_3.6}).

\item Any cone structure $\C$ is compatible with some Lorentz-Finsler metric $L$,  i.e., $\C$ is the cone structure of  $L$  (Cor. \ref{t_PRINNCIPAL}). 

\item The cone structures associated with two Lorentz-Finsler metrics $L_1,L_2$ coincide if and only if $L_1$ and $L_2$ are anisotropically conformal,  i.e., $L_2=\mu L_1$ for a smooth function $\mu>0$ defined on all timelike and causal vectors (Th.  \ref{t_anisotropic}). 

\end{enumerate}
\item  The cone geodesics of $\C$ are equal to the lightlike pregeodesics for any Lorentz-Finsler metric $L$ compatible with $\C$  (Th. \ref{t_CONE}). 

Therefore, all the natural causality theory for $\C$ becomes equivalent to the causality theory of any compatible $L$ 
(\S  \ref{s_62}).  
\end{enumerate} 
\end{thm}
 Here,  the ways to prove some of the previous results have a big interest in their own right, as 
they allow us to control in a precise way cone structures and to smoothen Lorentz-Finsler metrics.  Indeed, the following results are also obtained. 

\begin{thm}[Specification of cone structures and smoothability of Lorentz- Finsler metrics] Let $M$ be a manifold. Then:
\begin{enumerate}
\item Any cone structure $\C$ on $M$ yields a (non-unique) cone triple $(\Omega, T, F)$ composed by a non-vanishing 1-form $\Omega$, a vector field $T$ such that $\Omega(T)\equiv 1$ and a classical Finsler metric $F$ on the fiber bundle $\ker \Omega$ (Th. \ref{p_transversality}).   

\item Any cone triple $(\Omega, T, F)$ yields both, a smooth cone structure $\C$ (Th. \ref{p_transversality}) and a continuous Lorentz-Finsler metric   (naturally  defined as  $G(\tau,v)$ $=\tau^2-F(v)$) which is non-smooth only at the direction spanned by $T$ (Prop. \ref{l_G}).

\item  Under general hypotheses, a non-smooth  Lorentz-Finsler metric  can be smoothen maintaing the same associated cone   (Th. \ref{t_principal}) .

 In particular, the continuous Lorentz-Finsler metric $G$  associated with any cone triple can be smoothen by perturbing $G$ in a small neighborhood of the non-smooth direction $T$. 
\end{enumerate}
\end{thm}


Finally, we 
give an amount of examples of Finsler spacetimes. Indeed, we provide simple procedures of construction  by using
Riemannian and Finsler metrics, including a general construction of all Finsler spacetimes.  This can be summarized as follows. 


\bigskip \noindent  {\em Construction of new classes of (smooth) Lorentz-Finsler $L$.}  Let $M$ be a manifold and $\cF$, $g_R$ and $\omega$, respectively, a Finsler metric, a Riemmanian metric,  and a one-form on $M$.
\begin{enumerate}
\item {\em Basic examples}: $L(v)=\omega(v)^2-\cF(v)^2$ is a Lorentz-Finsler metric in the region $\bar A=\{v\in TM:\omega(v)\geq \cF(v)\}$, assuming that the indicatrix of $\hat F$ and $\omega^{-1}(1)$ intersect transversely (Th. \ref{t_examp}).
\item  {\em  Characterization in terms of   Riemannian and  Finsler metrics: }  $L(v)=g_R(v,v)-\cF(v)^2$ is a Lorentz-Finsler metric in the region $\bar A=\{v\in TM:g_R(v,v)\geq \cF(v)^2\}$, assuming that the relations \eqref{fundtensEx} and \eqref{fundtensEx2} below hold; moreover, any Lorentz-Finsler metric can be written in this way (Th. \ref{t_carac}).
\item {\em Stationary/static Finsler spacetimes:} natural examples can be constructed  in $\R\times S$ by considering metrics as above on a vector bundle over $S$;  a general local description  can also be obtained  (\S \ref{ex_stationary}).
\item {\em Smoothability:}   continuous Lorentz-Finsler metrics which appear naturally in $\R\times S$ when considering products or  generalizations of  static spacetimes can be smoothen preserving their static character  (part~(3) of Rem. \ref{r_principal}).

\end{enumerate}
 As a  further application of this study,  we suitably model and solve the problem of Zermelo navigation \cite{BCS04} when the  ``wind'' depends on time, that is,  when the  prescribed  (maximum)  velocities are given by the indicatrix of a Finsler metric which depends on  a parameter.  Indeed,  we define a cone triple  and take its  cone structure in such a way that   {\em the corresponding  cone geodesics provide the solutions of Zermelo problem }  (namely, the trajectories that minimize  the time under prescribed velocity in every direction).  Moreover, it is possible to ensure the existence of minimizers of the time under general conditions on the time-dependent Finsler metric  (Cor. \ref{c_Zermelo}).

 The article is organized as follows.  In Section \ref{s2}, cone structures are introduced. Even though some of the  issues  therein are somewhat elementary,  a detailed study is carried out to settle down some subtleties which become important later, as the following: (a)   to give  scarcely restrictive definitions of weak and strong cone (Def. \ref{d1}) and  to recover    properties of such cones (Prop. \ref{p_2.3}) which lie under the natural intuition (Ex. \ref{ex2.4}),  (b)~ the  introduction of a cone structure as a submanifold in $TM$ (Def. \ref{d_conestructure}) as well as its cone geodesics (Def. \ref{def_cone_g}) and (c) a useful  description of a cone structure by using cone triples (Th. \ref{p_transversality}). 
 
In Section \ref{s3}, the general background on Lorentz-Finsler metrics is introduced. We start at the basic notion of (properly) Lorentz-Minkowski norm on a vector space $V$, and prove a series of properties, including the existence of a natural smooth and convex cone (Lemma \ref{l(i)}, Prop. \ref{propiedades}). Then, the notion of Lorentz-Finsler metric is introduced and discussed (Def.~\ref{finslerst}, Prop.~ \ref{p_beem}), and the existence of an associated cone structure is proven (Cor.~\ref{c_3.6}). Moreover, Lorentz-Finsler metrics with the same cone are characterized as those {\em anisotropically equivalent} (Def. \ref{d_anisotropic}, Th. \ref{t_anisotropic}). 

In Section \ref{s4}, we focus on the construction of natural classes of smooth Lorentz-Finsler metrics. This has been a non-trivial issue in the literature. Indeed, some authors have included  the  existence of some non-smooth directions as a basic feature of Lorentz-Finsler metrics  (see Appendix \ref{s_a1}).  First, we provide a simple general construction by using a classical Finsler metric and a one-form (Th. \ref{t_examp} and its corollaries). Then,  the relativistic notions of {\em stationary} and {\em static} spacetime are revisited, including their explicit construction (Subsection   \ref{ex_stationary}). General procedures to construct new Lorentz-Finsler metrics from others with the same cone are developed in  Subsection  \ref{s4.2}.  Finally, the construction of all Lorentz-Finsler metric using  Riemannian and classical Finsler ones is  shown in Th. \ref{t_carac}

In Section \ref{s5},  we show first how, given any cone structure $\C$,  the choice of any cone triple $(\Omega, T,F)$ yields naturally a continuous Lorentz-Finsler metric $G$ which is smooth everywhere but on the timelike direction spanned by $T$ (Prop. \ref{l_G}). Then, we give a simple procedure to  smoothen  $G$ around $T$, obtaining  so a (smooth) Lorentz-Finsler metric with the same cone $\C$ (Th. \ref{t_principal}). As a consequence, each $\C$ is naturally associated with a class of conformally anisotropic Lorentz-Finsler metrics  (Cor.   \ref{t_PRINNCIPAL},   Rem.~\ref{rPRINCIPAL}). 

In Section \ref{s6}, after a brief summary on maximizing properties of geodesics for a Lorentz-Finsler metric (\S  \ref{s_61}),	the natural notion of {\em cone geodesic} is developed. Such geodesics (and, consistently,    their  conjugate or focal points) can be  computed  as lightlike pregeodesics for   any compatible Lorentz-Finsler metric $L$ (Th. \ref{t_CONE}); moreover, the computation can also be carried out by using the simple continuous Finsler metric $G$ associated with any cone triple (Rem.  \ref{r_CONE}). Their properties of minimization and extremality are stressed in  Subsection   \ref{s_63}. Indeed, they  yield  a simple solution to the extension of classical Zermelo navigation problem  to the time-dependent case  (Cor. \ref{c_Zermelo}).  Moreover, they permit to extend  naturally the properties of the geodesics in the so-called wind-Riemmannian structures to any wind-Finslerian structure  (\S \ref{s_632}). 

In Appendix \ref{s_a1}, our definition of Lorentz-Finsler spacetimes is compared with some others in the literature, and the possibility to apply our results to some of them is stressed. Finally, 
Appendix \ref{s_a2} includes a theorem summarizing the properties of Lorentz-Minkowski norms in comparison with classical norms, for the convenience of the reader.


\section{Cone structures}\label{s2}

Along this article, only real manifolds $M$ of finite dimension  $n$ will be considered. {\em Smooth} will mean $C^2$ when hypotheses of  minimal  regularity are considered (and consistently $C^1$, $C^0$ for the first and second derivatives of these elements), but it will mean ``as   differentiable as possible'' (including eventually $C^\infty$) consistently with the regularity of the ambient for the results of smoothability to be obtained.  Moreover, from now on, $V$ will denote any real vector $n$-space, $n\geq 2$. 

\subsection{Cones in a vector space}\label{conestructure} 
 We start by introducing some notions in  $V$. 
  As in \cite{ON}, our {\em scalar products} will be assumed to be only non-degenerate (but possibly indefinite); when convenient, $V$ is endowed with any auxiliary Euclidean (i.e., positive definite)  scalar product $h_V$.  A {\em domain} will be an open connected subset. 
 
 \begin{defi}\label{d1}   A smooth  hypersurface $\mathcal \C_0$ embedded in  $V\setminus\{0\}$  is a {\em  weak  (and   salient) 
 cone} when it satisfies the following properties:
 \begin{enumerate}[(i)]
 \item {\em Conic}: 
 for all $v\in  \C_0$, the {\em radial direction spanned by $v$},   $\{\lambda v:\lambda >0\}$, is included in $\C_0$. 
 
 \item {\em Salient}: if $v\in \C_0$, then $-v\notin \C_0$.
 
 \item {\em Convex interior}: $\C_0$ is the boundary in $V\setminus  \{0\}  $ of an open  subset $A_0 \subset V\setminus \{0\}$ (the $\C_0$-{\em interior}) which is convex,
 in the sense that, for any $u,w\in A_0$,  the segment $[u,w]:= \{\lambda u+(1-\lambda ) w: 0\leq \lambda \leq 1\}\subset V$ is included entirely in $A_0$;  in what follows, $\bar A_0$  will denote the closure of $A_0$ in $V\setminus \{0\}$, so that  
 $\bar A_0=A_0\cup \C_0$. 
  \end{enumerate}

\noindent  A weak cone is said to be a {\em strong cone} or just a {\em cone} when it satisfies: 
 
\begin{enumerate}[(i)]
 \item[(iv)] {\em (Non-radial) strong convexity}: the second fundamental form of $\mathcal{C}_0$  as an affine hypersurface of $V$  is positive semi-definite (with respect to an inner direction $N$ pointing out to $A_0$) and its radical at each point $v\in \C_0$ is spanned by the radial direction $\{\lambda v: \lambda>0\}$.
 \end{enumerate}
\end{defi}

\begin{rem}\label{r2.2}  There are some slight redundancies in  Def. \ref{d1}. 

(a)   First,  we assume explicitly that $\C_0$ does not include  the zero vector.  However, this could be deduced either from the stated definition of {\em salient} or, less trivially, by using the smoothness of $\C_0$ (the conic and salient properties would yield two half-lines containing $0$ in a way incompatible with smoothness).
Moreover, once $0$ is known to be excluded from $\C_0$, 
the hypothesis (ii)   
plus  the convexity  of $A_0$ in (iii)   can be replaced  just with  the convexity of $\bar A_0$.   Indeed, the hypothesis of being salient becomes trivial then, and the convexity of $A_0$ follows because it is the interior of $\bar{A_0}$ (recall, for example, the discussion around \cite[Def. 1.4]{BGS}). 

 (b) Less trivially, there is an overall relation between the notions of convexity for a domain ${\mathcal A}$ and its topological boundary  $\partial {\mathcal A}$. 
Namely, in general, for  any Riemannian 
metric on a manifold $M$, a  
   domain   $ {\mathcal A}\subset M$ such that its closure $\bar  {\mathcal A}= {\mathcal A}\cup \partial  {\mathcal A}$ is a complete  manifold with boundary satisfies: $ {\mathcal A}$ is  convex (in the sense that each two points can be connected by a minimizing geodesic) if and only if $\partial  {\mathcal A}$ is   infinitesimally convex (in the sense that the second fundamental form
   $\sigma^N$ of $\partial  {\mathcal A}$  with respect to  one, and then any direction $N$ pointing out to  $ {\mathcal A}$, is positive semi-definite), which holds even for regularity $C^{1,1}$  (see \cite[Th. 1.3]{BCGS});  recall that this convexity is less restrictive than {\em strong convexity}, which means positive definiteness. 
 In our case,  
the previous result cannot be applied directly to ${\mathcal A}=A_0$ because 
$\bar A_0$ is not a complete Riemannian manifold with respect to the auxiliary scalar product  $h_V$ and its topological boundary $\partial A_0$ in the whole $V$  is not smooth at $0$. However,  there are several ways to overcome this (see Example \ref{ex2.4} or Lemma \ref{c_strong} below).
In any case, for a strong cone,  the hypothesis~(iii) can be deduced from (iv)  just assuming that $\C_0\cup \{0\}$ is the topological  boundary of the   domain $A_0$ in $V$.  
\end{rem}

\begin{prop}\label{properties}
 For any  weak cone  $\C_0$ in $V$: 
\begin{enumerate}[(i)]
\item 
$A_0$ is conic and salient, 
\item  the topological boundary $\partial A_0$ of $A_0$ in $V$ is equal to $\C_0 \cup\{0\}$ and it is connected, 
\item  $\bar A_0$  is a smooth manifold with  boundary,
\item  $\C_0$ is closed in $V\setminus\{0\}$. 
\end{enumerate}
\end{prop}
\begin{proof}
To see that $A_0$ is conic, observe that,  otherwise,   the radial line containing some point $v$ of $A_0$ must contain a point
$w$ in  $\partial A_0\setminus \{0\}$,  but then $w$  would belong to $\C_0$ and, as this set is conic,  $v\in\C_0$, a contradiction with $(iii)$ in Def.~\ref{d1}. Moreover, $A_0$ is salient because it is convex and by definition $0\notin A_0$  concluding part $(i)$.   The first assertion of part~$(ii)$  follows from $(iii)$ in Def. \ref{d1} and the fact that $0$ belongs to the closure of $A_0$. Moreover, $\partial A_0$ is  (arc-)connected because it is conic (the boundary of a conic subset) and it contains the zero vector. To prove~$(iii)$,  observe that $\C_0$ is a smooth hypersurface and then  for every $p\in \C_0$, it divides  every small enough neighborhood of $p$, and (iv) follows because $\C_0$ is a boundary.  
\end{proof}
%
%
%
%
%
%
%

\begin{exe}\label{ex2.4}   Next, let us show a simple way to construct weak and strong cones, which will turn out completely general (Prop. \ref{p_2.3}). 
Let $\Pi$ be a hyperplane of $V$ which does not contain the zero vector. 
 When $n>2$, consider any   compact connected  embedded  hypersurface  (without boundary) $\s$ of $\Pi$, which is the boundary of an open  bounded  region 
 $\B$ of $\Pi$  by the Jordan-Brouwer Theorem.  
 Let $\C_0 \subset V$ (resp. $A_0\subset V$) be the set containing all the open half-lines  departing  from $0$  and meeting  $\s$  (resp.  $\B$).

Clearly, $\C_0$ is a weak cone (with interior $A_0$) if and only if $\s$ is infinitesimally 
convex with respect to $\B$ (thus diffeomorphic 
to an $(n-2)$-sphere), and $\C_0$ is a strong 
cone if and only if the second fundamental 
form of $\s$ is positive definite\footnote{\label{foot4} In 
this case,  when $n> 3$, $\s$ 
is  an {\em ovaloid}   (i.e. it is a  compact connected embedded hypersurface with positive sectional curvature)    of $\Pi$  by Gauss formula.   It 
is straightforward that any ovaloid is 
diffeomorphic to a sphere because its Gauss 
map yields a diffeomorphism (see, for example, \cite[VII. Th. 5.6]{KN}).  However, given any $n> 2$,  the result 
holds for $S_0$ 
 even when its second fundamental form  is only positive semi-definite.  Indeed, 
choosing any point $r_0\in \B$ all the half-lines starting at  $r_0$   in $\Pi$ must cross once 
and  transversely $\s$ (recall   the characterization of infinitesimal convexity 
in Rem.\ref{r2.2} (b)  and \cite[Prop. 3.2]{BCGS}),   providing then a diffeomorphism between $S_0$ and the ($n-2$)-sphere.  }.  In the case $n=2$, the role of $\s$ 
(resp. $\B$) can be played by any two distinct 
points $p,q\in \Pi$ (resp. the open segment with
endpoints $p,q$) and all the weak cones  become 
also strong ones.

\end{exe}

 The next technical result will be useful later;  it also stresses  the necessity of the compactness of $\s$ in the previous example.

\begin{lemma}\label{c_strong}
 Let $A_0$ be  any  connected  open conic salient  subset of $V$ 
 such that  its closure $\bar A_0$ in $V\setminus \{0\}$  is a smooth manifold with boundary 
 $\C_0:= \bar A_0\setminus A_0$. 
 Consider any affine hyperplane $\Pi\subset V$, with $0\not\in \Pi$, and the (vector) hyperplane $\Pi_0$ parallel to $\Pi$ through $0$.  
  The following properties are equivalent:
\begin{enumerate}[(i)]
\item  $\Pi$ is crossed transversely by all the radial directions $\{\lambda v:\lambda>0\}$ in $ \bar A_0$,
\item $\Pi$ is crossed transversely by all the radial directions $\{\lambda v:\lambda>0\}$ in $\C_0$,

\item 
when $n=2$, $\Pi\cap \C_0$ contains exactly two points; when $n\geq 3$, $\Pi_0$ does not intersect $\C_0$ and $\Pi\cap \bar A_0  \not=\emptyset$,
\item 
when $n=2$, $\Pi\cap A_0$ is an open (non-empty) segment; when $n\geq 3$, $\Pi_0$ does not intersect $ \bar A_0 $ and $\Pi\cap  \bar A_0  \not=\emptyset$. 
\end{enumerate}
 When these properties hold, then: 
\begin{itemize}
\item[(a)]   $\Pi\cap \bar A_0$ is  compact and 
$S_0:=\Pi\cap  \C_0$ is a compact embedded $(n-2)$-submanifold. 
\item[(b)]  
  $\C_0$ is a weak (resp. strong) cone with inner domain $A_0$ if and only if:  in dimension $n=2$, always; in dimension $n>2$, when $S_0$ is
 infinitesimally  convex (resp. strongly convex)  
 towards $A_0\cap \Pi$. 
 In  this case, $S_0$ is a topological sphere of $\Pi$ with non-negative  sectional curvature, and  it becomes an ovaloid  when $n>3$ and $\C_0$ is a strong cone. 

\end{itemize}
\end{lemma}

\begin{proof} 
When $n=2$, all the assertions follow easily by observing that $A_0$ must be a region delimited by two half-lines,  which is convex by the salient property; so, assume $n>2$. Clearly, $(i)\Rightarrow (ii)\Rightarrow (iii)$ is trivial. 
For $(iii)\Rightarrow (iv)$, observe that $(iii)$ implies that if $\Pi_0\cap A_0\not=\emptyset$, then $(\Pi_0\setminus\{0\})\subset A_0$ (recall that $\Pi_0\setminus\{0\}$ is connected), contradicting that $A_0$ is salient. For $(iv)\Rightarrow (i)$, observe that, assuming $(iv)$ and taking into account that $A_0$ is connected, it follows that $ \bar A_0$ must be contained in the open half-space determined by $\Pi_0$ which contains $\Pi$. 
 
 For (a), clearly  $S_0$ is an $(n-2)$ submanifold (the  transversal intersection of two hypersurfaces) and topologically  closed. The compactness of the sets $S_0$ and $\Pi\cap \bar A_0$ follows because,  otherwise, there would exist an affine half-line of $\Pi$ where a sequence of points of the corresponding set approaches. This would imply the existence of a half-line  of $\bar A_0$ contained in $\Pi_0$, in contradiction to $(iv)$. 
 
 For the  equivalence (b) recall first that, as $\Pi$ is totally geodesic in $V$, 
the second fundamental form $\sigma^N$ of $S_0$ in $\Pi$ is the restriction of the second fundamental form  $\tilde \sigma^N$ of  $\C_0$  in $V$  and, by conicity, $\tilde \sigma^N$ is semi-definite (resp. semi-definite with radical spanned by the radial direction) if and only if $\sigma^N$ is semi-definite (resp. definite).   
Therefore,  if $\C_0$ is a weak  (resp. strong)  cone then $\sigma^N$ is positive semi-definite  (resp. positive definite). 

 For 
the converse and the last assertion, let us  check first that $S_0$ is  connected. 
By conicity, the 
 (arc-)connectedness of $A_0$ implies 
that so is $\Pi\cap A_0$.
By the Jordan-Brouwer Theorem, 
each connected component of $S_0$ bounds an inner domain, 
 $\sigma^N$ can be positive 
semi-definite only towards its 
inner region and, thus,
 $A_0$ must lie always in the inner region delimited by each connected part of $S_0$.  So, if there were more than one part, 
 either one of them would  enclose another  (but $A_0$ would lie in the inner domain of the latter) or two connected 
parts with disjoint inner domains would exist 
(but $\Pi\cap A_0$ was connected).  

So, 
$S_0$ is connected, compact and embedded in $\Pi$; moreover,  its inner domain 
must be $\Pi \cap A_0$  
 (as the closure of this set is compact). Thus, the convexity of $\Pi \cap A_0$ (and, so,  $A_0$)  follows from its infinitesimal convexity, and the remainder
is straightforward  (recall footnote \ref{foot4} in Example \ref{ex2.4}). 
\end{proof}
 Next, all the cones are shown to be as the ones constructed in  Example~\ref{ex2.4}.  
\begin{prop}\label{p_2.3}
Let $\C_0$ be a  weak cone  with inner domain $A_0$.  Then: 
\begin{enumerate}[(i)]
\item there exists a vector hyperplane $\Pi_0 \subset V$ which does not intersect $\C_0$, 

\item for every hyperplane $\Pi_0$ as in (i)   and every linear form  $\Omega_0:V\rightarrow \R$ with $\ker \Omega_0=\Pi_0$, one of the  affine hyperplanes in $\{\Omega_0^{-1}(1), \Omega_0^{-1}(-1)\}$  intersects transversely all the  radial directions of $\C_0$  and $A_0$,  

\item  for any hyperplane $\Pi$ which intersects transversely all the  radial directions of $\C_0$, the intersection   $\s=\Pi\cap \C_0$   is  an infinitesimally convex   hypersurface of $\Pi$  diffeomorphic to an $(n-2)$-sphere. Moreover, $\C_0$ is a (strong) cone if and only if $S_0$ is strongly convex. The latter always occurs when $n=2$, and if and only if $S_0$ is  an ovaloid  (resp. $S_0$ is a curve with positive curvature) when $n>3$ (resp. $n=3$). 
\end{enumerate}
 So,   the   closure of $A_0$ in $V$    is a topological manifold with boundary $  \partial A_0= \C_0 \cup \{0\}$, 
$\C_0$ is connected if $n\geq 3$  and $\C_0$ is equal to two open half-lines starting at $0$ if $n=2$.    Moreover, $\bar A_0$ and $\bar A_0\cup\{0\}$ are convex. 
\end{prop}
\begin{proof}
 Consider the natural sphere  $S$ for the auxiliary scalar product $h_V$.  
 
 For  part {\em (i)}, take $\bar D_0=\bar A_0\cap S$ and its convex hull $CH(\bar D_0)$   in $V$, i.e., the smallest convex subset of $V$ containing $\bar D_0$ 
(intersection of all the convex subsets of $V$ containing $\bar D_0$). 
 Observe that the convex hull is equal to the subset of convex combinations of a finite number of points in $\bar D_0$, namely,
\[CH(\bar D_0)=\big\{ \sum_{i=1}^k\lambda_i v_i: k\in \N,  \text{$v_i\in \bar D_0$, $ \lambda_i\geq 0,  $ for $i=1,\ldots,k$; $\sum_{i=1}^k\lambda_i=1$ }\big\}.\] 
Let us see that $0\not\in CH(\bar D_0)$. 
Otherwise, there would exist a minimum finite number of points   $v_1, \dots , v_k\in \bar D_0$, $k\geq 2$, such that $0= \sum_{i=1}^k\lambda_i v_i$ with  $\lambda_i>0$  for all $i$.\footnote{\label{foot_carat} The existence of a finite number of points satisfying the stated property (and, thus, such a minimum number $k$) is well-known in convex theory  (recall the description given above of the convex hull).  Indeed,  one knows even $k\leq n+1$ (Caratheodory Theorem), but this inequality is not required here.} The convexity  and the conicity  of $\bar A_0  \cup \{0\} $  imply that $\hat v_1:= \sum_{i=2}^k\lambda_i v_i =-\lambda_1 v_1$ belongs to $ \bar A_0 $. As $0$ belongs to the segment $[v_1, \hat v_1]$ by construction, one of the following contradictions follows: (a) if $v_1,\hat v_1\in \C_0$, then $\C_0$ is not salient, (b) if $v_1,\hat v_1\in A_0$ then, by the convexity of $A_0$, $0\in A_0$,    or (c) otherwise (either $v_1\in \C_0$ and $\hat v_1\in A_0$ or the other way round), as $A_0$ is open and $\C_0\subset \partial A_0$,  there are two vectors $w_1,\hat w_1 \in A_0$ (arbitrarily close to $v_1,\hat v_1$) where the case (b) applies.
Recall that $CH(\bar D_0)$ is closed\footnote{Indeed, it is compact, as it is the convex hull of a compact subset in $V$ (this follows directly from Caratheodory Theorem, see footnote \ref{foot_carat}).} and, thus, there exists $ v_0\in CH(\bar D_0)$ such that  $\sqrt{h_V(v_0,v_0)} (> 0)$ is equal to the $h_V$-distance from $0$ to $CH(\bar D_0)$. So, if $\Pi$ is the affine hyperplane $h_V$-orthogonal to $v_0$ passing through $v_0$,  then  all $CH(\bar D_0)$ lies in the closure of the half-space of $V\setminus \Pi$ which does not contain $0$. Therefore, $\Pi_0$ can be chosen as the hyperplane parallel to $\Pi$ through $0$.



For {\em (ii)}, as $A_0$ is convex (thus, connected),   the whole  $\C_0$ is contained in one of the two open half-spaces in $V\setminus \Pi_0$. As $\C_0$ is conic, all its radial directions must intersect transversely   one of the hyperplanes $\{ \Omega_0^{-1}(1),\Omega_0^{-1}(-1)\}$  (the one which is included in that half-space).  

 Part {\em (iii)} follows from the last assertion in Lemma \ref{c_strong}. 

 For the last  assertions,  notice that the segments connecting $0$ and $S_0$ provide a topological chart around $\{0\}$,  concluding that $A_0$ is a topological manifold with boundary $\C_0\cup \{0\}$.  The convexity of $\bar A_0$ and $\bar A_0\cup \{0\}$ follows straightforwardly from the convexity of $S_0$. 
\end{proof}


\subsection{Cone structures and causality}\label{s2_con_est_causality}  Next,  previous notions are  transplanted      to manifolds. 

\begin{defi}\label{d_conestructure} Let $M$ be a manifold of dimension $n\geq 2$. A {\em (strong) cone structure}
 (resp. {\em weak cone structure})  $\mathcal{C}$ is an embedded hypersurface of $TM$ such that, for  each $p\in M$: 

\begin{enumerate}[(a)]
\item[(a)]  $\mathcal{C}$ is  tranverse to  the fibers  of the tangent bundle,  that is, if $v\in \C_p$  $:=$ $T_pM\cap \mathcal{C}$,  then $T_v(T_pM)+T_v\C=T_v(TM)$, and \item[(b)]  each $\mathcal{C}_p$
is a 
 strong cone  (resp. weak cone)  in $T_pM$ (as in Def. \ref{d1}).
\end{enumerate}
 The inner domain of each $\C_p$ will be denoted by $A_p$ and 
$A:=\cup_{p\in M}A_p$, which  will be called a  strong (resp. weak)  {\em cone domain}.  In the following, cone domains will always be assumed strong unless otherwise specified. 
\end{defi}

  Apart from the difference of being strong or weak, the terminology ``cone structure''  is used sometimes in a somewhat more general framework. For example,  in \cite[Def. 2.1]{Mak18}, there is  no assumption of convexity, as this reference focuses on local classification results by using Cartan's method of equivalence.
\begin{rem} 

The condition of transversality (a) also means that $\C$
is transverse to all the tangent spaces $T_pM, p\in M$  or, equivalently, $T_v(T_pM)\not\subset T_v\C$.

The intuitive role of transversality is the following. 
A cone structure $\C$ puts a cone at each tangent space in a seemingly smooth way, as $\C$ is smooth. However, one needs that the distribution of the cones  is a smooth set-valued function of $p\in M$, and this is grasped by our notion of transversality\footnote{If only  a  continuous distribution of cones  were required, then transversality would be interpreted at the topological level (compare with \cite{FS12}).}. 

The same property of transversality would be necessary for Riemannian or Finslerian metrics. Indeed, such a metric is determined by 
the  hypersurface $S\subset M$ formed by 
all
the indicatrices (unit spheres) $S_p$ for $p\in M$  (each $S_p$ being either an ellipsoid  centered at the origin $0_p\in T_pM$ or a  strongly  convex closed hypersurface enclosing $0_p$, respectively). However, this hypersurface $S$ must satisfy the condition of transversality (otherwise, the original 
metric would not be smooth), see \cite[Prop. 2.12]{CJS14}. The role of transversality will be apparent  in the proof of Th. \ref{p_transversality}.

\end{rem}

 A {\em Lorentzian metric} $g$ on a (connected) manifold $M$ is a symmetric bilinear form  with index one (signature $(-,+\dots , +)$). It is well known that its lightlike vectors (those $v\in TM\setminus \mathbf{0}$ with $g(v,v)=0$)  provide locally two (strong) cone structures (see Cor.  \ref{admitconest} to check consistency with our definition) and $g$ is called {\em time-orientable} when these cone structures  are globally defined; such a property becomes equivalent to the existence of a globally defined timelike vector field $T$ (i.e., $g(T,T)<0$), see \cite{MinSan,ON} for background. A {\em  (classical)   spacetime} is a time-orientable Lorentzian manifold $(M,g)$ where one of its two cone structures, called the {\em future-directed cone structure}, has been selected. The next definitions for cone structures generalize trivially those for classical spacetimes, even though we drop the expression ``future-directed'' as only one cone structure is being considered.


 Given a weak cone structure $\C$ there are two classes of privileged vectors at each tangent space:  the  {\em timelike} vectors, which are those in the cone domain $A$, and the {\em  lightlike} vectors, which are the vectors in $\C$; both of them will be called {\em causal}. This allows one to extend all the definitions in the Causal Theory, such as timelike, lightlike and causal curves and, then, the chronological $\ll$ and causal $\leq $ relations ($p\ll q$ if there exists a timelike curve  from $p$ to $q$; $p\leq q$ either if $\gamma$ can be found causal or if $p=q$), chronological $I^+(p)=\{q\in M: p\ll q\}$ and causal $J^+(p)=\{q\in M: p\leq q\}$ futures  for any $p\in M$, as well as the {\em horismotic} relation, namely: 
$p\rightarrow q$ if and only if $q\in J^+(p)\setminus I^+(p)$, for $p,q\in M$. Observe that the cone structure determines only the {\em future-pointing} directions, but one can say that a vector $v\in TM$ is {\em past-pointing} timelike (lightlike, causal)  if $-v\in TM$ is timelike (lightlike, causal) and, so, define analogous past notions. 
Consistently,  a {\em time} (resp. {\em temporal}) function is a real function $t:M\rightarrow\R$ which is strictly increasing when composed with (future-pointing) $C^1$-timelike curves (resp. a smooth time function $\tau$ such that no causal vector is tangent to the slices $\tau=$constant). Other conditions about Causality \cite{BEE, MinSan} as the notion of Cauchy hypersurface or being {\em strongly causal}, {\em stably causal} or {\em globally hyperbolic} are extended naturally. More subtly, $\C$ admits the following notion of geodesic which generalizes the usual lightlike pregeodesics of spacetimes.

\begin{defi}\label{def_cone_g} Let $\C$ be a  weak cone structure.
A continuous curve $\gamma: I\rightarrow M$ ($I\subset \R$ interval) is a {\em cone geodesic} if it is locally horismotic, that is, for each $s_0\in I$ and any neighborhood $V\ni 
\gamma(s_0)$, there exists a smaller neighborhood 
$U\subset V$ of $\gamma(s_0)$ such that, if $I_
\epsilon:=[s_0-\epsilon, s_0+\epsilon]\cap I$ 
satisfies $\gamma(I_\epsilon) \subset U$ for some $\epsilon>0$, then: 
$$
s<s' \Leftrightarrow \gamma(s)\rightarrow_U \gamma(s') \qquad \forall s,s'\in I_\epsilon, $$ where $\rightarrow_U$ is the horismotic relation for the natural restriction $\C_U$ of the cone structure to $U$.
\end{defi}

\begin{rem}
 Until now, the strengthening of the hypothesis  {\em weak cone} into {\em strong cone} has not been especially relevant. However, there will be important differences for geodesics,  which are  similar to the  standard  Finslerian case: if the indicatrix of a Finsler metric is assumed to be only infinitesimally convex (but not strongly convex), the local uniqueness of geodesics with each velocity is lost (see \cite{Matveev}). So, in what follows, we will focus only on the case of strong cones and strong cone structures, dropping definitively the word {\em strong}.
\end{rem}

\subsection{Pseudo-norms and conic Finsler metrics}  Even though the  notions  of Lorentz metric and spacetime will be extended to the Finslerian setting in Section \ref{s3}, 
next some basic language on pseudo-norms and Finsler manifolds are introduced.
\begin{defi} \label{d_psnorm}
A  function $L:A_0\subset V\setminus \{0\}\rightarrow \R$ is a {\em  (conic, two-homogeneous)  pseudo-Minkowski norm} if
\begin{enumerate}[(i)]
\item  $A_0$ is a  (non-empty)   {\em conic}  open subset (that is,  if $v\in A_0$, then $\lambda v\in A_0$ for every $\lambda>0$,  but $A_0$ is not necessarily salient), 
\item  $L$ is smooth and  {\em positive homogeneous of degree $2$}  (i.e., $L(\lambda v)=\lambda^2 L(v)$ for every $v\in A_0, \lambda>0$),  and
\item for every $v\in A_0$,  the  {\em fundamental tensor}  $g_v$ given by
\begin{equation}\label{fundtens}
 g_{v}(u,w)=\frac{1}{2}\left.\frac{\partial^{2}}{\partial  r  \partial  s} L(v+ r  u+sw)
 \right|_{ r  =s=0}
 \end{equation}
 for $u,w\in V$,  is nondegenerate.
\end{enumerate}
\end{defi}
The choice of being two-homogeneous becomes natural when the fundamental tensor is indefinite; however,  one-homogeneity will be required when convenient, namely:

\begin{defi} \label{d_conicnorm} A {\em conic Minkowski norm} is a  positive  function $F:A_0\subset V\setminus \{0\} \rightarrow \R^+  :=(0,\infty)$,   with $A_0$ open and conic,   and $F$ homogeneous of degree one (i.e., $F(\lambda v)=\lambda F(v)$ for every $\lambda >0$ and $v\in A_0$)  satisfying: the fundamental tensor $g_v$ in \eqref{fundtens} for $L=F^2$ is positive definite for every $v\in A_0$ (in particular, $L=F^2:A_0\rightarrow \R $ is a pseudo-Minkowski norm).
 
 Furthermore, when $A_0=V\setminus \{0\}$, then $F$ is a  {\em Minkowski norm}.  
\end{defi}

\begin{rem} 
A Minkowski norm   can be extended continuously to $0$ as  $F(0)=0$; this extension is always $C^1$, but it is $C^2$ if and only if $F$ is the norm associated with a Euclidean scalar product (see \cite[Prop. 4.1]{Warner65}). Such an extension will be used when necessary with no further mention.
\end{rem}

 Let us recover classical Finsler metrics consistently with our  definitions. 

\begin{defi}\label{d_ classical Finsler}  A {\em Finsler metric} on a manifold $M$ is a two-homogeneous smooth positive function $L:TM\setminus {\bf 0}\rightarrow \R$ with positive definite fundamental tensor $g_v$ in \eqref{fundtens} for all $v\in TM\setminus \bf 0$. 
When required, $L$  will be  replaced with $F=\sqrt{L}$ and extended continuously to the zero section $\bf 0$  (so that each $F_p:= F|_{T_pM}$ is a Minkowski norm). 

 An open subset $A^* \subset TM$ is {\em conic} when each $A^*_p := A^* \cap T_pM$, $p\in M$ is non-empty and conic; in this case, $A^*$ is a {\em conic domain} when each $A^*_p$ is also connected (and, then, a (strong) {\em cone domain} when the  additional conditions of  Def. \ref{d_conestructure} are also fulfilled).
 When $L$ above  satisfies all the properties of a Finsler metric but it is defined only an open conic subset $A^*\subset TM$, we say that $L$ is a conic {\em Finsler metric}. \footnote{\label{foot_coniccone}  Even though conic Finsler metrics are defined here in arbitrary open conic subsets, here we emphasize the notions of {\em conic domain} and {\em cone domain} to be used later. Tipically,    we will select a conic domain as a connected part of a conic open set and, when a Lorentz-Finsler metric is defined on such a domain, we will prove that it is a cone domain  (see Rem. \ref{r3.6} and Prop. \ref{p_beem}).  These subtleties should be taken into account when comparing with references on the topic.   }

\end{defi}
 This definition is extended trivially to any vector bundle $VM$ (in particular, to any subbundle of $TM$) in such a way that a Finsler metric  on $VM$ becomes a smooth distribution of  Minkowski norms  in each  fibre of the bundle.

\subsection{Cone triples} Next, a natural link between cone structures and some triples which include a Finsler metric is developed.

\begin{lemma}\label{l_2.1} Given a cone structure $\C$, one can find on $M$: 

(a) a {\em timelike} 1-form $\Omega$ (that is, $\Omega(v)>0$ for any causal vector $v$), 

(b) an $\Omega$-unit timelike vector field $T$ ($T$ is  timelike  and $\Omega (T)\equiv 1$). 
\end{lemma}

\begin{proof} By the definition of cone, one can find at each point $p$ a one-form $\omega_p$ such that $\omega_p(\C _p)>0$  (recall Prop. \ref{p_2.3}).  By continuity (just working in coordinates) one can regard $\omega_p$ as a 1-form defined in a neighborhood $U_p$ of $p$ and satisfying $\omega_p(\C _q)>0$ for all $q\in U_p$. Now, consider  a locally finite open refinement $\{U_{p_i}:i\in \N\}$ of $\{U_p:p\in M\}$ and  a partition of unity $\{\mu_i: i\in \N\}$ subordinated to $\{U_{p_i}:i\in \N\}$. The required one-form is just 
 $\Omega= \sum_{i=1}^{+\infty} \mu_i \omega_{p_i}$. 
Once $\Omega$ is constructed, let $\tilde T$ be any timelike vector field (constructed analogously by using a partition of unity and the convexity of the cones) and put $T=\tilde T/\Omega(\tilde T)$.
\end{proof}

\begin{rem} The 1-form $\Omega$ is neither exact nor closed in general. However, from the proof is clear that, locally, $\Omega$ can be chosen exact, so that $\Omega=dt$ for some smooth $t:U (\subset M) \rightarrow \R$. In this case $t$ is naturally a {\em temporal function} for the restriction $\C_U$ of the cone structure to $U$. 
\end{rem}

 Any pair $(\Omega, T)$ associated with $\C$ according to Lemma \ref{l_2.1} yields a natural  splitting of $TM=$ span$(T)\oplus $ 
$\ker$ $\Omega$ with  projection $\pi: TM\rightarrow $
$\ker$ $\Omega$ determined trivially by:
\begin{equation} \label{e_OmegaDecomposition}
v_p= \Omega(v_p) \, T_p + \pi(v_p) \qquad \forall v_p\in T_pM, \; \forall p\in M.
\end{equation}

A close link between Finsler metrics and cone structures is the following.

\begin{thm}\label{p_transversality} Let  $\C$ be a cone structure. For any choice of timelike 1-form $\Omega$ and $\Omega$-unit timelike vector field  $T$, there exists a unique (smooth) Finsler metric $F$ on the vector bundle $\ker$ $\Omega\subset TM$ such that, for any $v_p\in TM\setminus
\mathbf{0}$
\begin{equation}\label{e_F}
v_p\in \C \Longleftrightarrow v_p=F(\pi(v_p))T_p+\pi(v_p).
\end{equation}

Conversely, for any {\em cone triple} $(\Omega,T,F)$ composed of a non-vanishing one-form $\Omega$, an $\Omega$-unit vector field $T$ and a Finsler metric $F$ on $\ker$ $(\Omega)$, there exists a (unique) cone structure $\C$ satisfying \eqref{e_F}; such a $\C$ will be said {\em associated with the cone triple}.
\end{thm}

\begin{proof} Let us check that $\Sigma^F:= \pi(\Omega^{-1}(1)\cap \C )$ satisfies all the properties for being the indicatrix of the required Finsler metric. Both, $\C$ and $\Omega^{-1}(1)$ are smooth hypersurfaces of $TM$ (transversal to all the fibers of $TM$)  which intersect transversely; thus, as $\Omega$ is timelike, $\Omega^{-1}(1)\cap \C $ is an embedded $(2n-2)$-submanifold transversal to the fibers of $TM$. 
These properties are shared by $\Sigma^F$, because it is obtained as a  pointwise translation\footnote{The translation by $T$ can be regarded as a change in the zero-section for the associated affine bundle and, so, cannot affect the claimed transversality, since it is a diffeomorphism that preserves the fibers.}, namely,  $\Sigma^F= (\Omega^{-1}(1)\cap \C) -T $. Recall that, by construction,  each $\Sigma^F_p:=\Sigma^F \cap T_pM$ is a compact strongly convex hypersurface included in $\ker$ $\,(\Omega_p)$ which encloses $0_p$ and, so, it defines a (1-homogeneous)  Minkowski norm $F_p: \hbox{$\ker$ } \Omega_p$ $\rightarrow [0,+\infty)$. So, it is enough to show that $F:$ $\ker$ $(\Omega) \rightarrow [0,+\infty)$, $F(v_p):=F_p(v_p)$ for all $v_p\in $  $\ker$  $(\Omega_p), p\in M$, is smooth away from the zero section. Now, consider the map:
$$
\psi: (0,\infty) \times \Sigma^F \rightarrow \hbox{ $\ker$ }(\Omega)\setminus \mathbf{0}, \qquad (r,w)\mapsto r\cdot w .
$$
Clearly, this map is bijective and smooth. Moreover, its differential is bijective at all the points. 
 Indeed, putting $\partial_r=(1,0)\in T_{(r,w)}((0,+\infty)\times \Sigma^F$),  one has $d\psi(\partial_r)$ is proportional to the position vector and then  transversal to\footnote{For the role of tranversality, see \cite[Section 2.2]{CJS14}, especially Rem. 2.9 and the proof of Prop. 2.12.} $r\cdot\Sigma^F$.
Therefore, $\psi$ is a diffeomorphism and, by construction, its inverse satisfies 
$\psi^{-1}(v)=(F(v),v/F(v))$; thus, $F$ is smooth, as required. 

For the converse, the unit sphere bundle $\Sigma^F$ for $F$ is a smooth submanifold in $TM$ transverse to each $T_pM$, and so is its (pointwise translation) $T+\Sigma^F$ and its conic saturation $\C$. Moreover, $(T+\Sigma^F)\cap T_pM$ is strongly convex in $\Omega^{-1}(1)\cap T_pM$ for every $p\in M$  and the construction in Example \ref{ex2.4} applies.
\end{proof}

\begin{rem} It is clear from Lemma \ref{l_2.1} that  a cone structure  yields  many cone triples, while one cone triple determines a unique cone structure.  In the case that $T$ is complete and $\Omega$ exact,  
 $\Omega=dt$ for some function $t$, then $M$ splits as $\R\times S_0$, where  $S_0$ is the slice  $\{t=0\}$, $t$ becomes the projection onto the first factor and $T\equiv \partial_t$. Indeed,  the splitting is 
 $\R\times S_0\rightarrow M$, $(s,x)\rightarrow \Phi_s(x)$, where $\Phi$ is the flow of $T$, because $t(\Phi_s(x))=s$ (as $dt(T)\equiv 1)$, it is a local diffeomorphism (as $T$ and the slices of $t$ are transversal), it is injective (as no integral curve of $T$ can cross $S_0$ twice) and onto (as the integral curve of any $p\in M$ must cross $S_0$ because of the completeness of $T$).   
Notice that, locally, one can always choose an exact $\Omega$; so, around any $p\in M$, one has an analogous splitting $(t(p)-\epsilon, t(p) +\epsilon) \times N$ for  some neigborhood $N$ of $p$ in the slice $\{t=t(p)\}$ and $\epsilon>0$. 
\end{rem}

A straightforward consequence is the following. 

\begin{cor}\label{admitconest} The set of all the future-directed lightlike vectors of a classical spacetime forms a cone structure according to Def. \ref{d_conestructure}.
 Moreover,  a manifold $M$ admits a cone structure if and only if $M$ is non-compact or  its Euler characteristic  is 0. 
\end{cor}

\begin{proof}  For the first assertion,  the   Lorentzian metric $g$  of a spacetime  admits a  unit future-directed timelike vector field $T$;
so, the set of all the future-directed $g$-lightlike vectors is the cone structure associated with the triple $(\Omega,T,F)$, where $\Omega$ is the 1-form $g-$associated with $T$ and $F$ is the  norm of the  Riemannian metric obtained as  the  restriction of $g$ on  $\ker$ $(\Omega)=T^\perp$ (the subbundle $g-$orthogonal  to $T$). 

 For the last assertion,    the existence of a vector field $T$ on $M$  which is  non-zero  everywhere becomes equivalent to either the condition on the  Euler characteristic or the non-compactness of $M$, \cite[Prop. 5.37]{ON} (see also \cite{Kokkendorff}). 
 Then, the implication to the right follows because Lemma \ref{l_2.1} ensures  the existence of such a $T$ and, for the converse, construct a time-oriented Lorentz metric \cite[Prop. 5.37]{ON} and consider the set of all its future-directed lightlike vectors\footnote{ Alternatively, use Th. \ref{p_transversality}, namely: take a non-vanishing vector field $T$, construct any auxiliary Riemannian metric $g_R$ on $M$, define $\Omega$ as the 1-form $g_R$-associated with $T$, and choose $F(v)$ as the restriction of $\sqrt{g_R(v,v)}$ to $v\in $  $\ker$ $(\Omega)$.  }.   
\end{proof}

Given two cone structures $\C, \C'$ on $M$, denote $\C  \preceq  \C'$  if the cone of $\C'$ is included in the one of $\C$ (say, $\C'\subset \bar A$). So, we have  the following simple consequence of Th. \ref{p_transversality}.

\begin{cor}\label{c_abriendo_conos}
 Given a  cone structure $\C$, there exist two Lorentzian metrics $g_1, g_2$ such that their cone structures  $\C_1, \C_2$ satisfy 
  $\C_1 \preceq \C \preceq \C_2 $. 

\end{cor}

\begin{proof} Take a cone triple $(\Omega, T,F)$ and any Riemannian metric $h_R$ on  $\ker$ $(\Omega)$.  Multiply $h_R$ by some big enough (resp. small enough) conformal factor $e^{2u_1}$ (resp. $e^{2u_2}$) such that the unit sphere bundle  of $h_i:=e^{2u_i}h_R$ is included in  (resp. includes)   the  indicatrix of $F$ pointwise.  Then each $\C_i$  is just the cone structure determined by 
	$(\Omega, T, F_i=\sqrt{h_i})$.
\end{proof}

\begin{rem}\label{r_fathi}
We are focusing on smooth cone structures instead of more general  ones. 
 Indeed,  the main differences of our definition of cone structure and the one of Fathi and Siconolfi in \cite{FS12} are the differentiability of $\C$  and the strong convexity of each $\C_p$, which are not required in that 
reference\footnote{More general cone structures  in \cite{Suhr}  drop  continuity and  allow  singular cones.}. 
However, the notion of cone triple would make sense for such general cone structures and cone triples would characterize them just taking into account that the Finsler metric $F$ on  $\ker$ ($\Omega$) would become now a continuous distribution of norms whose regularity would depend on the assumptions of regularity and convexity  for the cones. 
\end{rem}

\section{Finsler spacetimes}\label{s3}
\subsection{Lorentz-Minkowski norms and their cones}  Let us start  with  notions  at the level of a vector space $V$,  consistently with  \cite{JavSan11}. 
\begin{defi} \label{d_LMnorm}
A pseudo-Minkowski norm $L:  A_0  \subset V\setminus 0 \rightarrow\R^+$ is  {\it Lorentz-Minkowski} if   $A_0$ is 
 non-empty, connected,  conic and open,  and the fundamental tensor in  \eqref{fundtens} has index $n-1$;  in this case,    when there is no possibility of confusion,  the one-homogeneous function $\F=\sqrt{L}$ will  be also considered  and called  {\em Lorentz-Minkowski norm}.  

Moreover,  $L$ is  {\em properly Lorentz-Minkowski} if, in addition, the  topological boundary $\C_0$ of $A_0$ in $V\setminus 0$  is smooth (i.e., 
 $\bar A_0:= A_0\cup \C_0$  is a smooth manifold 
with boundary)  and $L$  can be smoothly extended as zero to     $\C_0$   with non-degenerate fundamental tensor  (then, the same letters $L,\F$ will denote such extensions).  In this case, we will also write $L:\bar A_0\rightarrow \R$  and, when required,   the continuous extension $L(0)=0$ will also be assumed. 
\end{defi}

\begin{rem}
Observe that there are some cases  in the bibliography where  the pseudo-Finsler metric has index $n-1$, but it cannot be extended smoothly to the boundary  (see for example \cite[Prop. 2.5 (iii) (b)]{CJS14}  or a translation of a Lorentzian metric, which, naturally, can be continuously extended  to the boundary as zero, but not smoothly, \cite[Prop. 2.9]{JV}  or the alternative definitions in Appendix \ref{s_a1}).  Here,  we will be  interested in the proper Lorentz-Minkowski case. 
The next proposition will show, in particular, that $\C_0$ must be a  (salient, strongly convex, with convex interior)  cone  for any properly Lorentz-Minkowski norm.   
\end{rem}
The following lemma will provide a criterion for the smoothness of $\C$ and  it will be also useful for other purposes.   

\begin{lemma}\label{l(i)}
 Let $ A^*_0 \subset V$ be a  non-empty   connected conic open subset  and $L:  A^*_0 \rightarrow \R$  be a  pseudo-Minkowski norm with index $n-1$. Assume that $A_0:=L^{-1}((0,\infty ))$ is connected and  its topological closure  $\bar A_0$ in $V\setminus\{0\}$ is included in  $ A^*_0$.
 Then:
 $$g_v(v,v)=L(v), \qquad dL_v(w)=2 g_v(v,w), \qquad \forall v\in  \bar A_0, \; \forall w\in V,$$
  where $g_v$ was defined in \eqref{fundtens}. 
Therefore,  $\bar A_0$  is a smooth manifold with boundary $\C_0=L^{-1}(0)   \setminus \{0\}$  (and so $L|_{A_0}$ is a properly Lorentz-Minkowski norm)  and the {\em indicatrix} $\Sigma_0:=L^{-1}(1) \subset A_0$ is a smooth hypersurface. 
\end{lemma}
\begin{proof}
The equalities follow   for any pseudo-Minkowski norm as in the case of Minkowski norms and Finsler metrics  (see for example \cite[Prop. 2.2]{JavSan11}).  As a consequence, $0$ and 1 are regular values of $L$ (up to the origin) and 
$\C_0, \Sigma_0$ become smooth.  
\end{proof}

\begin{prop}\label{propiedades}
 Let  $L:\bar A_0\subset V\rightarrow \R$  be a  properly  Lorentz-Minkowski norm and  $\C_0$, $\Sigma_0$, as above.
 Then:
 \begin{enumerate}[(i)]
 \item For any $v\in \Sigma_0$,  the restriction   of the fundamental tensor  $g_v$ to $T_v\Sigma_0$ 
 (which can regarded as the $g_v$-orthogonal space to $v$)  is negative definite. 
 \item $\Sigma_0$ is  connected and  strongly convex with respect to the position vector. 
   Moreover,  let $S$ be any ellipsoid\footnote{ We consider ellipsoids as they are intrinsic to the vector space structure of $V$; alternatively, spheres for the auxiliary Euclidean scalar product $h_V$ can  also be considered. }
 centered at $0$, and consider the  functions
$\lambda: A_0\cap S\rightarrow \R^+$, $\lambda(v):=1/\F(v)$ and $\phi:A_0\cap S\rightarrow \Sigma_0$, $\phi(v)=\lambda(v)v$.  Then,
$\Sigma_0$ is asymptotic to $\C_0$ in the  sense that  $\C_0$ is conic and  $\lambda(v)\rightarrow +\infty$ whenever $v\rightarrow  \C_0 \cap S$.
 \item For every $v\in \C_0$,  the tangent space $T_v\C_0$ is the $g_v$-orthogonal space to $v$ and the restriction of  $g_v$ to $T_v\C_0$  is negative semi-definite,  being the direction of $v$  its only degenerate direction.
 \item  The second fundamental form $\sigma^\xi$ of $\C_0$  with respect to any vector  $\xi\in T_vV$ 
   pointing to $A_0$   is positive semi-definite  with radical spanned by $v$.
 \item Given any smooth extension of $L$ with non-degenerate fundamental tensor,   
 its domain contains an open subset $A_0^* \supset   \bar A_0  $ 
such that  $L<0$ in $A_0^*\setminus \bar A_0$ (for computations around $\C_0$, such a subset can be regarded as the domain of the extension of $L$ to $\C_0$.) 
 \item  $\C_0$ is a  strong  cone (according to Def. \ref{d1}) with $\C_0$-interior $A_0$. 
 
 \item  Given any $v\in \Sigma_0$,  the intersection $T_{v}\Sigma_0\cap\C_0$ is a strongly convex hypersurface in $T_{v}
\Sigma_0$ diffeomorphic to a sphere. 
 Given any $v\in \C_0$, the 
intersection  $T_v\C_0 \cap  \bar A_0$ is the half-line $\{\lambda v: \lambda \geq  0\}$. 
 \end{enumerate}
 \end{prop}
\begin{proof}
 For $(i)$,  Lemma \ref{l(i)} implies that 
$T_v\Sigma_0$   is given by the $g_v$-orthogonal vectors to $v$. In particular, as $g_v$ has index $n-1$ and $g_v(v,v)=L(v)>0$, the fundamental tensor is negative definite in $T_v\Sigma_0$. 

 For part $(ii)$, 
recall first that, easily, 
\begin{equation}\label{secfundform}
g_v(X,X)=- 
\sigma^v(X,X)v(L)/ 2, 
\end{equation}
where $\sigma^v$ is the second fundamental form of $\Sigma_0$ with respect to the position vector  $v$  and $X\neq 0$ is a tangent vector to $\Sigma_0$  at  $v$ (see \cite[Eq. (2)]{AaJa16} and also \cite[Eq (2.5)]{JavSan11}). So, the strong convexity of $\Sigma_0$ follows from \eqref{secfundform}, 
 taking into account that  its left hand side is negative by  part $(i)$ and  $v(L)>0$ by  positive homogeneity.  Now observe that by positive homogeneity and using that $S$ and $\Sigma_0$ are transversal to the radial directions, $\phi$ is a diffeomorphism. Moreover,  $A_0\cap S$ is connected because,  otherwise, $A_0$ would not be; as a consequence, $\Sigma_0$ is also connected. 
By homogeneity, $\C_0$  is conic, and 
 $\Sigma_0$ is asymptotic to $\C_0$ because, otherwise, $L$ could not be extended (not even continuously) by 0 to $\C_0$.

For part $(iii)$, repeat the reasoning in part $(i)$ taking into account that if $L(v)=g_v(v,v)=0$, then  $v$ is a  lightlike vector of $g_v$  and its $g_v$-orthogonal space  must contain the direction spanned by $v$ (recall  Lemma \ref{l(i)}); thus, this direction must  be the unique degenerate direction allowed by   the Lorentzian signature $(+,-,\ldots,-)$ of $g_v$ (see \cite[Lemma 5.28]{ON}). 
  
 For $(iv)$, reasoning as in    \eqref{secfundform}, one has 
\begin{equation}\label{secondfundxi}
g_v(X,X)=-\sigma^\xi(X,X)\xi(L) /2.
\end{equation}
 So, the result holds from   $\xi(L)=2g_v(v,\xi)>0$. To prove the latter,  first,  $\xi(L)\geq 0$ since $L$ is zero in the boundary  with  $L>0$ in $A_0$;  then,  the equality cannot hold 
because $\xi$ is not tangent to $\C_0$ (recall part $(iii)$). 

Part $(v)$ is a consequence of 
Lemma \ref{l(i)}, since now $-\xi(L)=-2g_v(v,\xi)<0$ (as in the reasoning of part $(iv)$) for any  $-\xi$  pointing away from $A_0$.

 For the remainder, notice that part  {\em (vi)} follows if there exists  a hyperplane $\Pi\not\ni 0$ which is  crossed transversely by all the radial half-lines of $\bar A_0$ (use then the last assertion of Lemma~\ref{c_strong}, taking into account that  by  part {\em (iii)} above,  $\Pi\cap \C_0$ is strongly convex). We are going to prove that such a $\Pi$ can be
chosen as $T_v \Sigma_0$ for any $v\in \Sigma_0$, which proves additionally the first assertion in part $(vii)$ (by using again Lemma~\ref{c_strong}). 
 Take $w\in T_v\Sigma_0$  and consider the $2$-plane $P={\rm span}\{v,w\}$. Observe that $L|_P$ is also Lorentz-Minkowski  and its indicatrix $\Sigma_P$ is a strongly convex curve by part $(ii)$. If we choose a positive definite scalar product such that  $w$ and $v$ are orthonormal, and coordinates $(x,y)$ in this basis, it turns out that $\Sigma_P$ can be parametrized in polar coordinates in terms of the angle as it is not tangent to the radial lines. Moreover, when $\theta=\pi/2$ its slope is zero as $w$ is tangent to it, and when $\theta$ decreases, because of the strong convexity, the slope of $\Sigma_P$ increases. As $\Sigma_P$ cannot be tangent to radial lines, by continuity, its slope remains below some $\alpha>0$. This implies that the ray $v+\lambda w$, $\lambda>0$, meets the cone $\C$ transversely and this gives a diffeomorphism from the sphere to $\Sigma_0\cap \C$ as required.

 For the  last assertion in part {\em (vii)},
the conicity of $\C_0$ implies that the radial line spanned by $v$   lies in the tangent space ($v\in \ker dL_v$) 
and part $(iv)$ of Def. \ref{d1},  that no more points can appear in the intersection.
\end{proof}

\subsection{Lorentz-Finsler metrics}\label{s3.2}
In the literature there are several definitions of Finsler spacetimes. Let us give first the definition which, from our viewpoint,  has better mathematical properties.
\begin{defi}\label{finslerst}
Let $M$ be a manifold and $TM$ its tangent bundle. Let $A\subset TM\setminus 0$ be  a conic domain (according to Def. \ref{d_ classical Finsler}) such that     its closure $\bar A$ in $TM \setminus 0$  is an embedded smooth manifold with boundary.
Let  $\C  \subset TM \setminus 0$  be its  boundary  and $L:A\rightarrow \R^+$ a smooth function which can be smoothly extended as zero to $\C$ satisfying,  for all $p\in M$, that 
$$L_p:=L|_{A_p}, \qquad \hbox{where} \qquad A_p:=A\cap T_pM,
$$ 
is a properly Lorentz-Minkowski norm.   
  Then, $L$    
will be called a {\em Lorentz-Finsler metric}, and $(M,L)$  a \emph{Finsler spacetime};  when necessary, $L$ will be assumed  continuously extended to
 the zero section $\mathbf{0}\subset TM$ (and denoted with the same letter). 
\end{defi}

\begin{rem} \label{r3.6}  (1) 
Even if $A_p$ is not required to be convex and salient,  both properties follow from part $(vi)$ of Prop. \ref{propiedades} (in particular,   the definition above  coincides with the one given in \cite{JavSan14}).

 (2)  As $L$ is smooth on $\C$ with non-degenerate fundamental tensor,   $L$ can be extended to an open conic subset $A^*$  containing $\bar A$ such that the fundamental tensor of $L$ has  index $n-1$ on $A^*$ and  $L<0$  in $A^*\setminus \bar A$  (this is just a  straightforward generalization of part $(v)$ of Prop. \ref{propiedades}; say, the local result would follow trivially, and the global one by using a partition of  unity).  Clearly,  such an  $A^*$ can also be chosen as a conic domain.

 Even if $L$ is defined beyond $\bar A$, our definition of Lorentz-Finsler metrics prescribes $A$ and, then, the cone structure $\C$. So,   
 all the concepts of Causality Theory in  Section \ref{s2} apply here and  Finsler spacetimes are always time-oriented. 
\end{rem}

\begin{cor}\label{c_3.6}  
If $L: A\rightarrow  \R^+$ is a Lorentz-Finsler metric, then  the boundary $\C$ of $A$ in $TM\setminus \bf 0$  is a cone structure with  cone domain   $A$  
($\C$ and $A$  will be called {\em  associated with} $L$, or just  
 the {\em cone structure} and {\em cone domain}  of $L$). 
\end{cor}
\begin{proof}
  By part {\em (vi)} of Prop. \ref{propiedades}, each $\C_p=\C \cap T_pM$, $p\in M$, is a strong cone, while 
 Lemma \ref{l(i)} implies that  $\C$ is transverse to all $T_pM$. 
\end{proof}

 In classical Beem's definition \cite{Beem70}, Lorentz-Finsler
metrics are  defined in the whole tangent bundle. 
Clearly,  our results will be applicable to such metrics  whenever a cone structure is fixed.  Implicitly, this assumes time-orientability;  more precisely: 

\begin{prop}\label{p_beem} 
Let $A^*$ be a domain of $TM$ such that each $A^*_p := A^*\cap T_pM$ is conic and non-empty. Let $L: A^*\rightarrow \R$ be a two-homogeneous smooth function whose fundamental tensor $g$ (as in  \eqref{fundtens}) has index $n-1$. 

 Assume that there exists a non-vanishing vector field $X$ in $A^*$ ($X_p\in A^*$ for all $p\in M$), such that $L(X)>0$. If $A$ is the connected part of $L^{-1}((0,\infty))$ containing $X$ and its closure  $ \bar A $ in $TM\setminus \mathbf{0}$ is included in $A^*$,  then $L$ is a Lorentz-Finsler metric with cone domain  $A$. 
\end{prop}
\begin{proof}  Observe that  Lemma \ref{l(i)} guarantees that every $L_p=L|_{A\cap T_pM}$ is a properly Lorentz-Minkowski norm,  so it is enough to check that   $\bar A \subset TM\setminus{\bf 0}$  is a smooth manifold with boundary, which follows because its boundary
$\C$ is the inverse image of 
the regular value $0$ of $L$
 (use again that $g_v$ is non degenerate for all $v\in \C$ and
   $dL_v(w)=2 g_v(v,w)$).
\end{proof}
 As a difference with Beem's approach \cite{Beem70},  we will focus all our attention on  $\bar A$ considering properties of $L$ independent of possible extensions.  

\subsection{Anisotropic equivalence} 

In  order to characterize the Lorentz-Finsler metrics with the same associated cone structure, let us introduce  the following natural extension of a concept for classical Finsler metrics. 
 \begin{defi}\label{d_anisotropic}
 Two Lorentz-Finsler metrics $L_1, L_2:\bar A\rightarrow [0,+\infty)$ are said to be {\em anisotropically equivalent} if there exists a smooth positive function $\mu:  \bar A \rightarrow \R^+$    such that $L_2=\mu L_1$;  then, the  function\footnote{ Necessarily 0-homogeneous and, thus, non-continuously extendible to the zero section~$\bf 0$. } $\mu$ is called the {\em  anisotropic factor}. 
 \end{defi}
The following lemma will be useful to characterize this definition as well as to study other properties in  Subsection  \ref{s4.2}. 
\begin{lemma}\label{lsmooth}  Let $L_1, L_2$ be two smooth functions on a manifold $N$  
and let $\C$ be a hypersurface obtained as $\C=L_1^{-1}(0)=L_2^{-1}(0)$, where $0$ is a common regular value of $L_1, L_2$. If $\C$ is the boundary of a domain $A$ where $L_1,L_2>0$, then  $L_2/L_1$ can be smoothly extended to $\C $  (as a positive function)  and $\sqrt{L_1L_2}$ is smooth on $\bar A$. 
\end{lemma}

\begin{proof} We can assume that, locally, $L_1$ is the first coordinate $r=x_1$ of a chart $(x_1=r,x_2,\dots , x_n)$ around some $p\in\C$ ($r$ can be thought as the distance function to $\C$ for 
the auxiliary Riemannian metric  $g_R=\sum dx_i^2$). The local function $\mu:=L_2/r=L_2/L_1$ on $A$ can be smoothly extended  on
$\C$ as $\partial_r L_2$  $ = 
dL_2(\partial_r)>0$ 
(recall that $\partial_r$ is transversal to $\C$ and points out inside $A$), proving the first assertion.  Then, one has locally on $\bar A$
$$
\sqrt{L_1L_2}=r\sqrt{\mu},
$$
 where the right-hand side  is the product of two smooth functions, as $\mu$ does not vanish on $\C$.
\end{proof}

\begin{thm}\label{t_anisotropic}
Two Lorentz-Finsler metrics $L_1,L_2:\bar A\rightarrow [0,+\infty)$ are anisotropically equivalent if and only if their associated cone structures are equal.  

Moreover,  in such a case the smooth extension to any $v\in \C$ of the factor of anisotropy $\mu=L_2/L_1$ on $A$ can be computed as  
\begin{equation}\label{e_mu}
 \mu(v)=\frac{g^2_v(v,w)}{g^1_v(v,w)},
\end{equation}
where $g^1$ and $g^2$ are the fundamental tensors of $L_1$ and $L_2$, respectively,  and  $w$  is any vector in $T_{\pi(v)}M$ such that $g^1_v(v,w)\neq 0$ (and, thus, $g^2_v(v,w)\neq 0$). 
\end{thm}
\begin{proof} $(\Rightarrow)$ Obvious from the definition.

$(\Leftarrow)$ Notice that 
$L_2/L_1$ (which is smooth on $\C$ by  Lemma~\ref{lsmooth} 
applied locally to  a neighborhood $N$ of 
each point of the common cone)  provides the 
anisotropic factor.  In order to  
   check \eqref{e_mu}, 
 given $v\in\C$ and $w\in T_{\pi(v)}M$ 
as stated, we can assume that, for small $|t|>0$,  $v+tw$ belongs to the open domain $A^*_{\pi(v)}\supset A_{\pi(v)}$ of some extension of $L_1,$ and $L_2$.  By  applying L'H\^opital rule and  Lemma \ref{l(i)},  
\[ \lim_{t\rightarrow 0} \frac{L_2(v+tw)}{L_1(v+tw)}=\frac{g^2_v(v,w)}{g^1_v(v,w)}.\]
Finally, observe that $g^1_v(v,\cdot)$ and $g^2_v(v,\cdot)$ are one-forms with the same kernel (the tangent space to the lightlike cone), and then the quotient does not depend on $w$. 
\end{proof}


Finally, we emphasize that the tangent bundle $T\C$ can also be characterized in terms of the fundamental tensor $g$ of  any compatible $L$.  Recall that for a classical Finsler metric $F$ with fundamental tensor $g$, its Hilbert form is defined as $\omega(w)=g_v(v,w)/F(v)$. In the Lorentz-Finsler case, such a form does not make sense (as one would divide by 0), however, expressions as \eqref{e_mu} show that a similar form may have interest.

\begin{prop} Let $L$ be a Lorentz-Finsler metric  with fundamental tensor $g$ and $\C$ its associated cone structure.  Consider the {\em rough Hilbert form}  $\omega^L: \bar A  \rightarrow TM^*$ 
 defined as
$\omega^L_v=g_v(v,\cdot )$, 
for all  $v\in \bar A.$  Then 
$$
T_v\C = \hbox{ $\ker$ }(\omega_v^L) \qquad \forall v\in \C
$$
\end{prop}
\begin{proof} Apply  part $(iii)$ of  Prop. \ref{propiedades}. 
\end{proof}

\begin{rem}  From \eqref{e_mu}, 
if $L_2=\mu L$ is a second  Lorentz-Finsler metric with cone $\C$, then $\omega^{L_2} =  \mu \omega^L$  on $\C$, consistently  with the fact that $T\C$ is associated with the anisotropically conformal class of Lorentz-Finsler metrics compatible with $\C$.
\end{rem}

\section{Constructing new examples of (smooth) Finsler spacetimes}\label{s4}
   It seems that a systematic construction of  (smooth)  Finsler spacetimes  as  above  is missing in literature,   being the only examples we  have found  either  perturbations of classical spacetimes \cite{Voicu17} or  anisotropically conformal to Lorentz metrics \cite[Eq. (5)]{Min15}. In this section, we will try to fill this gap  by  characterizing all possible examples  and  constructing easily some families. 
  \subsection{A natural class of Finsler spacetimes} \label{newclass}
  Next, new examples of smooth Finsler spacetimes will be constructed. 


\begin{thm}\label{t_examp}
Let $M$ be a manifold endowed with a   (classical)  Finsler metric $\cF: TM \rightarrow \R$  with indicatrix $\hat\Sigma=\cF^{-1}(1)$ and a  
non vanishing one-form  $\omega$ such that, at each point $p\in M$, the intersection $\hat\Sigma_p\cap \omega_p^{-1}(1)$ is (non-empty) transverse. 
Then  $L: \bar A \rightarrow \R$ defined as  
\begin{equation}\label{e_ejemplo}
L(v):=\omega(v)^2-\cF(v)^2 \qquad \quad \forall v\in \bar A:=\{w\in TM \setminus \mathbf{0}: \,  \omega(w) \geq  \cF(w)\}
\end{equation} 
is a Lorentz-Finsler metric  with cone domain $A$ equal to the interior of $\bar A$. 

 Moreover, the cone structure $\C$  of $L$ can be described by a cone triple $(\Omega, T, \hF )$ with $\Omega= \omega$, $T$ any vector field in $\omega^{-1}(1)\cap A$ and $\hF$  the Finsler metric on $\ker \Omega$ with indicatrix $(\hat\Sigma \cap \omega^{-1}(1))-T$ (i.e., the translation with $-T$ of $\hat\Sigma \cap \omega^{-1}(1)\subset TM$). 
\end{thm}
\begin{proof}  First, notice that $A$ is convex, since it is the conic saturation of a convex subset (the intersection of $\omega^{-1}(1)$ and the unit ball of $\cF$). So, one can find a vector field $X$ in $A$ in a standard way (first locally and, using a partition of unity, globally) and choose the normalized one $T=X/\omega(X)$. So,  we have just to prove that the fundamental tensor $g_v$ of $L$ has index $n-1$ for all  $v\in \bar A$ and claim Prop. \ref{p_beem} with $A^*=TM$ (this implies that  $L$ is Lorentz-Finsler and the remainder is straightforward).  
Observe that 
\[g_v(u,w)=\omega(u)\omega(w)-\hat g_v(u,w)\]
where $v\in TM\setminus \bf 0$, $u,w\in TM$ and $\hat g_v$ is the fundamental tensor of $\cF$.  Trivially, $g_v$ is negative definite in the hyperplane $\ker (\omega)$. As   $L(v)=g_v(v,v)$ by  positive homogeneity, if $v\in A$ then $g_v(v,v)=\omega(v)^2-\cF(v)^2>0$, and the required index is obtained. So, assume otherwise  that $v\in \bar A$ and   $L(v)=g_v(v,v)=0$.   
As $\hat\Sigma$ and $\omega^{-1}(1)$ are transversal, we can choose a vector $w\in \ker(\omega)$  not tangent to $\hat\Sigma$.  Necessarily,  $\hat g_v(v,w)\not=0$, and then
\[w(L)=g_v(v,w)=-\hat g_v(v,w)\not=0\]
 So,  $g_v$ restricted to  span$\{v,w\}$ has Lorentzian signature, there exists  $u\in $ span$\{v,w\}$ with $g_v(u,u)>0$, and the index of $g_v$ becomes again $n-1$. 
\end{proof}

 Up to a re-scaling, the previous result can be applied to any $F$ and $\omega$. 

\begin{cor}\label{c_examp}
 If   $(M,\cF)$ is a Finsler manifold and $\omega$ a  non-vanishing  one-form on $M$, there exists a positive  function $\mu:M\rightarrow  \R $ such that  $$L(v_p)  :=  \left(\mu(p)\omega(v_p)\right)^2-\cF(v_p)^2$$  
  for every $ v_p\in \bar A:=\{w_p\in  TM\setminus {\bf 0} : \,  \mu(p)\omega(w_p) \geq  \cF(w_p)\}$ 
is a Lorentz-Finsler metric.
\end{cor}
\begin{proof} In some neighborhood $U_p$ around each $p\in M$, one  can take $\mu>0$ big enough so that all the intersections $\hat \Sigma_q  \cap (\mu\omega_q)^{-1}(1), q\in U_p$, are transverse.  By means of a partition of unity, $\mu$ can be chosen globally and   Th. \ref{t_examp} applies.
\end{proof}
 It is worth pointing out that the previous procedure may yield Lorentz-Finsler metrics even in the case that they are not naturally extendible to all the tangent bundle,  that is, when $\cF: A^* \rightarrow \R$ is just a conic Finsler metric according to Def. \ref{d_ classical Finsler}. The only caution now is that, in order to apply Prop. \ref{p_beem} we have to ensure that $\hat\Sigma_p \cap \omega_p^{-1}(1)$ is not only transverse but also compact (so that $\bar A \subset A^*$), that is: 
\begin{cor} Let $\cF: A^* \rightarrow \R$ be a conic Finsler metric and $\omega$ a one-form such that each $ \hat \Sigma_p  \cap \omega_p^{-1}(1)$ is non-empty, transverse and compact. 
Then  $L: \bar A \rightarrow \R$ defined as in \eqref{e_ejemplo} is a Lorentz-Finsler metric. 
\end{cor}
 A particularly interesting example of conic Finsler metris are Finsler-Kropina ones, defined as a quotient, 
$$
 F_0^2  /\beta : A^*\rightarrow \R, \quad v\mapsto F_0(v)^2/\beta(v), \qquad \forall v\in A^*:= \{w_p\in TM:  \beta(w_p) >0  \},
$$
 where $F_0$ is a Finsler metric  and $\beta$ a non-vanishing one-form on $M$ 
(see \cite[Cor. 4.2]{JavSan11}). This is a classical Kropina metric when $F_0$ comes from a Riemannian metric.  An extension of Cor. \ref{c_examp} is then: 
\begin{cor}
 Let  $F_0^2/\beta:  A^*\rightarrow \R^+ $ be a Finsler-Kropina metric and  $\omega$ a one-form on $M$  such that at no point $p\in M$ the equality $\omega_p=  \lambda  \beta_p$ holds for $\lambda\leq 0$. 
Then there exists a positive function $\mu:M\rightarrow \R^+$ such that  
$$L(v_p)=\left(\mu(p)\omega(v_p)\right)^2-\left(F_0(v_p)^2/\beta (v_p)\right)^2 
$$
for all $v_p\in \bar A:= \{w_p\in  A^*   : \, \mu(p)\beta(w_p)\omega(w_p) \geq  F_0^2(w_p)  \}$  defines  a Lorentz-Finsler metric.
\end{cor}

\begin{proof}
The indicatrix $ \hat \Sigma_p  $ of $F_0^2/\beta$ at each $p$ is a strongly convex  hypersurface and, moreover,  $\{0_p\} \cup  \hat \Sigma_p $ is a compact hypersurface which lies on one side of  $\ker \beta$.  Thus, if $\omega_p$ is not proportional to $\beta_p$, its kernel $\omega_p^{-1}(0)$ intersects $\Sigma_p$ transversely and, for big $\mu>0$, the intersection $\hat \Sigma_p  \cap (\mu\omega_p)^{-1}(1)$ is both, transversal and compact; clearly, this also holds when
 $\omega_p = \lambda \beta_p$ for $\lambda> 0$. So, the result follows as in 
 Cor.~\ref{c_examp}.
\end{proof}
Th. \ref{t_examp} can be used in several situations as the following.

\begin{exe}[Perturbations of classical Lorentz metrics] Let $(M,g_L)$ be a time-orientable Lorentz manifold $(-,+,\dots , +)$ with associated Lorentz-Finsler metric $L(v)=-g_L(v,v)$ (recall that we assume $L$ positive on the timelike directions). Choosing any timelike unit vector field $T$, one can define the Riemannian metric  
$$g_R(v,w):=g_L(v,w)-2 g_L(v,T)g_L(w,T)/g_L(T,T).$$ In terms of the one-form $\omega(v)=\sqrt{2}\, g_L(v,T)/\sqrt{|g_L(T,T)}|$, one has:
\[g_L(v,w)=g_R(v,w)-\omega(v)\omega(w), \qquad \hbox{i.e.} \quad L(v)=\omega(v)^2-g_R(v,v).\]
This last expression can be seen as a particular case of  Th. \ref{t_examp} taking $g_R$ as $F^2$. Small perturbations of  $g_R$ will transform it into a Finsler metric whose indicatrix retains the conditions of transversality and compactness in that theorem, yielding a 
Lorentz-Finsler metric 
 not associated with a classical Lorentz metric (compare with \cite[\S 5.A]{HP17}). 

Such perturbations can be obtained in several ways. 
 For example, a Randers perturbation can be obtained by adding a one-form $\mu \tilde\omega$
(i.e., $L(v)=\omega(v)^2-(\sqrt{g_R(v,v)}+ \mu\tilde\omega(v))^2$) where,  once the one-form 
$\tilde \omega$ is prescribed, the function  $\mu>0$ is  chosen  small enough  to make Th. \ref{t_examp} applicable.   
More generally, for any Finsler metric $\cF$ and small $\mu>0$ we can add $\mu \cF$ (even relaxing the positive definiteness of its fundamental tensor into  
positive semi-definiteness,  recall \cite[Th. 4.1]{JavSan11}), that is, 
$$L(v)=\omega(v)^2-(\sqrt{g_R(v,v)}+\mu \cF(v))^2.$$ Such an $\cF$ is arbitrary and can be generated, for example with norms of the type $\cF(x_1,x_2,\ldots,x_n)=\sqrt[r]{x_1^r+\ldots+x_n^r}$, for even $r$  (it is not difficult to check that its fundamental tensor is positive semi-definite).
\end{exe}

\subsection{Stationary  and static  Finsler spacetimes} \label{ex_stationary}
A Finsler spacetime is {\it stationary} when it admits a timelike Killing vector field $K$  (also called {\em stationary}),  where {\em Killing} means that its (local) flow 
preserves\footnote{ Given two Finsler spacetimes $(M,L), (M',L')$ a isometry $\phi: M\rightarrow M'$  is a diffeomorphism which preserves the metrics ($\phi^*L'=L$) and, then, the corresponding cones ($\phi_*\C=\C'$).
In particular, the flow of a Killing vector field preserves the cone structure and, so, it is also an {\em anisotropically conformal} vector field in a natural sense (recall Th. \ref{t_anisotropic}).
See \cite[\S 2.9]{Aniso16}
for further descriptions in terms of a Lie derivative.  }  $L$.   A stationary  Finsler spacetime is {\it static with respect to the timelike Killing field $K$}  (which is then also called the {\em static vector field})  if its orthogonal distribution $K^\perp$ is integrable, where
\begin{equation}\label{e_ortog}
K^\perp=\{ v\in TM:  g_K(K,v)=0\},   \end{equation}
being $g$  the fundamental tensor of $L$.  Clearly,  if $L$ is static with respect to $K$, then it is also static with respect to $\lambda K$,  whenever $\lambda$ is a positive constant.

In the case that $L$ is  stationary and it comes from a classical Lorentz metric $g$, this can be written locally as a {\em standard stationary spacetime}, 
\begin{equation}\label{e_classical_stationary}
g_{(t,x)}=-\Lambda(x) dt^2 + 2\alpha_x dt +(g_0)_x \qquad (t,x)\in  (a,b)\times S
\end{equation}
where, with natural identifications, $\Lambda>0, \alpha$ and $g_0$ are, resp.  a function, a one-form and a Riemannian metric on the factor $S$ of the local product $M \equiv (a,b)\times S$ 
and $K\equiv \partial_t$; moreover, if $g$ is static with respect to $K=\partial_t$, then it 
can be written as {\em a standard static 
spacetime} i.e., as above with $\omega\equiv 0$ 
(see for example \cite[Chapter 12]{ON}  or \cite[\S 7.2]{SachsWu}). 
Such a standard expression has been used sometimes  to generalize classical spacetimes into Lorentz-
Finsler ones just replacing the metric $g_0$ in \eqref{e_classical_stationary} with a Finsler one $F_0$ (see \cite{CaStan16}). 
 However, such Lorentz-Finsler metrics share the lack of smoothability of product Finsler manifolds (in our case, they are not smooth on the section $\R \times {\bf 0}$ of the tangent bundle,  see Prop. \ref{l_G});   this fact motivates the following subsection.

\subsubsection{A simple construction of  (smooth) stationary  Finsler spacetimes.}  Th. \ref{t_examp}  allows us to construct easily smooth Lorentz-Finsler metrics which are stationary or static, according to our (natural) definition. Namely, consider the product manifold 
$M=\R\times S$, 
the fiber bundle $\pi^*_M(S)$ over $S$ obtained as the   pull-back   of $\pi:TM\rightarrow M$ through the inclusion $i:S\rightarrow M,  x\mapsto (0,x)  $:
\begin{equation}\label{fiberbundle}
    \xymatrix{ \pi_M^*(S) \ar[d]_{\pi^*_M} & TM \ar[d]^{\pi} \\
               S \ar[r]^{i}  & M  }
\end{equation}
and take  any classical Finslerian metric $\cF$ and one-form $\omega$ in the bundle $\pi^*_M(S)$ such that $\omega(\partial_t) >\cF(\partial_t) $ (this condition can be ensured just by starting with any $\tilde \omega$ which does not vanish on $\partial_t$ and re-scaling to $\omega:=\mu \tilde \omega$, where $\mu= 2 \cF(\partial_t)/\tilde \omega(\partial_t)$).
Now define, on $\bar A\subset TM$: 
\[L_{(t,x)}(v)=  \omega_x(v_0)^2-\cF_x(v_0)^2,  \; 
\forall v\in \bar A:=\{w\in  TM \setminus {\bf 0} :\omega( w_0  )^2-\cF(w_0)^2\geq 0\}, 
\]
where,  for any $v\in T_{(t,x)}M$,  $v_0$ denotes the tangent vector at $(0,x)$ (and thus, in the pulled-back bundle) obtained   by 
 moving $v$ with the flow of $\partial_t$. 
As the conditions in   Th. \ref{t_examp} are fulfilled,  $L$ becomes a Lorentz-Finsler metric, which is stationary by construction.  Easily, $L$ is also static if $TS$ is both, the Kernel of $\omega$ and the $g_{\partial_t}$-orthogonal of $\partial_t$. Notice that the property ``the  $g_K$-orthogonal space to $K$ must be the tangent space to $S$'' can be interpreted geometrically as ``the tangent space to the indicatrix of $\cF$ at $K$ is parallel to the tangent space to $S$''.

\subsubsection{General local characterization and constructions.}\label{s_characterzitation static}  Next, the local characterization  \eqref{e_classical_stationary} of any  classical stationary spacetime, will be properly generalized to the stationary Lorentz-Finsler case. 
Given the Killing vector field $K$ and $p_0\in M$, choose any hypersurface $S$  with compact closure embedded in $M$, transverse to $K$ which contains $p_0$,  and use the flow of $K$ to smoothly split $M$ as $(-\epsilon,\epsilon)\times S$ around $p_0$  for some $\epsilon>0$.  Then,  one can define the Lorentz-Finsler metric on the fiber bundle $\pi_M^*:\pi^*_M(S)\rightarrow S$ introduced in \eqref{fiberbundle} as  
$L_{(t,x)}(v)=L^S_{x}(v_0)$, where now $L^S$ is the pullback metric from $L$ on   $\pi^*_M(S)$. Choosing $S$ not only transversal to $K$ but also to its  cone $\C$ (i.e. $TS\cap \bar A=\emptyset$),   the cone $\C^S$ of $L^S$ can be described with a triple $(\Omega,T,F)$ where  $T=K$ and $\Omega= dt$ (being $t: (-\epsilon,\epsilon)\times S \rightarrow (-\epsilon,\epsilon)$ the natural projection). In the static case, $S$ can  also be chosen as an integral manifold of $K^\perp$.

 It is worth pointing out that a similar construction  allows one to construct  locally any stationary or static Finsler spacetime on $M$, in an explicit way.   Namely, 
  as in the last paragraph, consider a (precompact) hypersurface $S$ and a transverse vector field $K$ in such a way that $M$  splits as $(-\epsilon,\epsilon)\times S$ in a smooth way. Then any Lorentz-Finsler metric $L^S$
  on the fiber  bundle\footnote{Observe that $L^S$ can be constructed by any of the procedures described along the present paper for  the construction of $L$ on the whole $M$, including the general procedure for $L$ in Th. \ref{t_carac} below, as emphasized in Rem. \ref{r_fibrado}.}   $\pi_M^*:\pi^*_M(S)\rightarrow S$ with $K|_S$ 
  in its domain $A$ and such that  $TS\cap \bar A=\emptyset$, 
  can be extended to a Lorentz-Finsler metric on $(-\epsilon,\epsilon)\times S$ using the flow of $K=\partial_t$, namely, $L_{(t,x)}=L^S_x$. 
Moreover, in order to construct a static Lorentz-Finsler metric on $(-\epsilon,\epsilon)\times S$, we can proceed as follows. For each $x\in S$, the strong convexity of the indicatrix $\Sigma_x^S$ of $L^S$ implies that there is a unique point $u_x \in \Sigma_x^S$ such that the hyperplane $T_{u_x}\Sigma_x^S$ (tangent to the indicatrix at $u_x$) is parallel to $T_xS$. Then, $K$ will be static if (and only if) each $K_x$ is in the half-line spanned by $u_x$ for every $x\in S$.  In  particular, given any Lorentz-Finsler metric $L^S$ on the fiber bundle with  $TS\cap \bar A=\emptyset$, one can choose $K_x=u_x$ for all $x\in S$ and, then, $K$ will be unit and static.

\subsubsection{ Standard stationary and static Finsler spacetimes}  The previous constructions on $(-\epsilon,\epsilon)\times S$ can be extended  trivially  to $\R\times S$ by using the flow of $K=\partial_t$.  This justifies   the following generalization of the  notion 
of {\em standard} stationary or static for classical spacetimes, avoiding   problems of smoothability in  the formal extension of the expression \eqref{e_classical_stationary}. 

\begin{defi} A {\em standard stationary} Finsler spacetime is a product manifold $M=\R\times S$ endowed with a Lorentz-Finsler metric $L$ such that the  natural vector field $\partial_t \; (\equiv (1,0))$ is stationary  and  $\bar A$ (determined by its cone structure $\C$)  does not intersect the distribution induced by $TS$ on $M$.

Moreover, when this distribution is equal to the orthogonal one  $\partial_t^\perp$ (computed as in \eqref{e_ortog}), the Finsler spacetime is also called  
 {\em standard static}.
	\end{defi}
Notice that the construction in the second part  of  Subsection  \ref{s_characterzitation static} provides a way to generate (all) standard stationary and static spacetimes.
Moreover,   
the characterization of stationarity provided in the first part of that  subsection 
can be summarized as follows:

\begin{prop}
Every stationary (resp. static) Finsler spacetime is locally isometric to a standard stationary (resp. standard static) one.
\end{prop}

 Finally, recall that the  preservation of the cone $\C$ occurs naturally for conformal fields
(see \cite{Voicu18} for a recent study).   This   leads naturally to the notion of conformastationary Finsler spacetime (extending  the classical metric case),  where the  ideas introduced above can also be applied.


\subsection{New examples from anisotropically conformal ones 
}\label{s4.2}
 Trivially, new  examples of Lorentz-Finsler metrics can be obtained from one, $L$,  by means of an anisotropically conformal change, i.e., multiplying $L$ by a suitable positive smooth   0-homogeneous  function $\mu: \bar A\rightarrow  \R$.  In order to ensure that $\mu$ is suitable as an isotropic factor,  
 $\mu$ can be  chosen,    for example,  as a function which is  $C^2$-close enough  to a constant function $c>0$.
Next, we will see that  further  new examples can be obtained  by combining different Lorentz-Finsler metrics with the same cones  and using pseudo-Finsler metrics and one-forms.  We 
will do this  by extending to pseudo-Finsler metrics a general  result in \cite[Th. 4.1]
{JavSan11} for Finsler metrics  and by using     their   angular metrics. 

In the following, 
 $A^*$ will be a conic domain 
 and 
$\tF_k: A\rightarrow  (0,\infty), k=1,\dots , \n$,  smooth  {\em positive} one-homogeneous functions. 
Even though we will apply our results to the case when all $\tF_k$ are Lorentz-Finsler, this condition will not be imposed a priori. So, we will say that such an $\tF_k$ is  {\em pseudo-Finsler}, emphasizing that the corresponding fundamental tensor $g^k$  defined in \eqref{fundtens}
 might  become  degenerate  
 (that is, so may be
the fundamental tensor  $g^k_v$ of $\tF_k$ at the tangent vector $v\in A^*$). This generality allows a better comparison with results in the Finslerian case. 
The so-called {\em angular metrics} (see \cite[Eq. 3.10.1]{BaChSh00}) are determined by
\begin{equation}\label{angularmetric}
h_v^k(w,w)=g_v^k(w,w)-\frac{1}{\tF^2_k(v)}g_v^k(v,w)^2, \qquad  \forall v\in A^*, w\in T_{\pi(v)}M. 
\end{equation}
Let $\beta_{\n+1},\beta_{\n+2},\ldots,\beta_{\n+\m}$
 denote $\m$ one-forms on $M$.  The indexes $k,l$  will run from
$1$ to $\n$ while the indexes $\mu, \nu$ will label the  one-forms  and run from $\n+1$ to
$\n+\m$; the indexes $r, s$ will run from $1$ to $\n+\m$. 
Let $B$ be a conic open subset of $\R^{\n+\m}$ and consider a
continuous  function $\varphi: B\times M\rightarrow \R$, which
satisfies:
\begin{itemize}
\item[(a)] $\varphi$ is smooth and positive away from $0$, i.e., on $(B\times
M) \setminus (\{0\}\times M)$.

\item[(b)] $\varphi$ is $B$-positively homogeneous of degree 2, i.e.,
$\varphi(\lambda x, p)=\lambda^2 \varphi(x,p)$ for all $\lambda>0$ and all
$(x,p)\in B\times M$.
\end{itemize}
The comma will denote derivative with respect to the corresponding
coordinates of $\R^{\n+\m}$, namely, we will denote by $\varphi_{,rs}$  the
second partial derivative of $\varphi$ with respect to the $r$-th and
$s$-th variables.
Finally, consider the function $L: A^*\subseteq TM\rightarrow \R$
defined as: 
\begin{equation} \label{ef2}
L(v)=\varphi(\tF_1(v), 
\ldots,\tF_\n(v),\beta_{\n+1}(v),
\ldots,\beta_{\n+\m}(v),\pi(v)).
\end{equation}
\begin{prop}\label{central}
For any $\varphi$ satisfying (a) and (b) as above, and $L$ as in
\eqref{ef2}, the function $L$ 
is a pseudo-Finsler  metric with domain $A^*$ and fundamental
tensor:
\begin{multline}
 2 g_v(w,w)
 =\sum_k\frac{\varphi_{,k}}{\tF_k(v)} h^k_v(w,w)
 + \sum_{k,l}\frac{\varphi_{,kl}}{ \tF_k(v)\tF_l(v)}g_v^k(v,w)g_v^l(v,w)
\\+2\sum_{k,\mu}\frac{\varphi_{,k\mu}}{\tF_k(v)}g_v^k(v,w)\beta_\mu(w)+\sum_{\mu,\nu}\varphi_{,\mu\nu}
\beta_\mu(w)\beta_\nu(w).\label{fundamentalTensor2}
\end{multline}
\end{prop}
\begin{proof}
It is obtained in an analogous way  to formula (4.7) in \cite{JavSan11}.
\end{proof}

 Now, let us focus in the Lorentz-Finsler case. Recall that,  in this case,  $L=\tF^2$ can be extended to $\bar A^*$ 
 but the angular metric cannot. 
 So, as a previous algebraic question:

 \begin{lemma}\label{lorentzchar}
  Let $g$ be a symmetric bilinear form on $V$  
admitting  a hyperplane $W$ such that $g|_{W\times W}$ is negative semi-definite with radical of dimension at most 1.  If  there exists $w\in V\setminus W$ such that $g(w,w)>0$ and $g(w,v)\not=0$ for all $v\not=0$ in the radical of $g|_{W\times W}$,  then  $g$  is non-degenerate with  index $n-1$.
 \end{lemma}
 \begin{proof} 
 Assume that the radical of $g|_{W\times W}$ is spanned by some $v\neq 0$ (otherwise, the result is trivial). Thus, $W=v^\perp$ (the orthogonal of $v$ in $V$) and, by the assumptions on $w$, the plane  $P:= $ span$\{w,v\}$ has Lorentzian signature. Clearly,  $P^\perp  \subset  W\setminus\{v\}$  (so, dim($P^\perp$) $= n-2$ and $P^\perp$ is non-degenerate) and $P\cap P^\perp=\{0\}$.  Then, $V=P\oplus P^\perp$  and the result follows. 
 \end{proof}
 
\begin{prop} \label{indexn-1}
	 If  $L:A\rightarrow \R^+$  is a Lorentz-Finsler metric, with cone $\C$, then
  
\begin{enumerate}[(i)]
\item for each  $v\in  A$,  the angular metric 
\[h_v(u,w)=g_v(u,w)-\frac{1}{L(v)}g_v(v,u)g_v(v,w) \qquad \forall  u,w\in   T_{\pi(v)}M ( \equiv   T_v(T_{\pi(v)}M) ) \] 
is negative semi-definite with radical spanned by $v$, and 

\item   for each $v\in \C$, the one-form $\omega_v=g_v(v,\cdot)$  on $T_{\pi(v)}M$ is non-trivial and 
 the restriction of $g_v$ to $\ker \omega_v$ is negative semi-definite with radical spanned by $v$.
\end{enumerate}
In this case, $\omega_v(w)>0$ for all $\C$-causal vector $w$ independent of $v$.

Conversely, let  $A$ be  a connected open conic subset and $L:A\rightarrow  \R^+$ a (positive  2-homogeneous)  pseudo-Finsler metric  smoothly extendible   as zero to the boundary $\C$ of $A$ in $TM\setminus \bf 0$. If the pseudo-Finsler metric $L$ satisfies $(i)$ and $(ii)$   and there exists  a non-vanishing vector field $X$ contained in $A$, then $L$  is a Lorentz-Finsler metric. 
\end{prop}
\begin{proof} 
To check $(i)$, clearly, $v$ belongs to the radical and $h_v$ is negative definite on the $g_v$-orthogonal space to $v$, as the index of $g_v$ is $n-1$ and $h_v=g_v$ there.  For $(ii)$, apply part $(iii)$ of Prop. \ref{propiedades} and observe that
 the causal vector $w\in T_{\pi(v)}M$ points to the interior of the cone structure, thus $g_v(v,w)=dL_v(w)=w(L)>0$.
  For the converse,  notice  that $A$ cannot intersect the zero section, as $L$ is positive and 2-homogeneous.  Let us see that $g_v$ has index $n-1$. When $L(v)>0$, $h_v=g_v$ in the $g_v$-orthogonal space to $v$, and then $g_v$ is negative definite there;  as  $g_v(v,v)=L(v)>0$,  necessarily, $g_v$ has index $n-1$.  When $L(v)=0$,  Lemma \ref{lorentzchar} yields the nondegeneracy of $g$ at $\C$  and Prop.   \ref{p_beem} concludes.
\end{proof}
 The last part of this proposition can be applied to  pseudo-Finsler metrics, as in
the following consequence (such a result is less trivial than expected even in the classical Finsler case, compare with  \cite[Cor. 4.3]{JavSan11}).  First we will need a technical result. 
\begin{lemma}\label{l_suma}
	Let $ L_1,\ldots, L_\n:   A\rightarrow \R  $  be pseudo-Finsler metrics on $M$ with fundamental tensor possibly degenerate.  Then if $L=(\varepsilon_1\sqrt{L_1}+\ldots+\varepsilon_\n\sqrt{L_\n})^2$, where $\varepsilon_i^2=1$ for $i=1,\ldots,\n$,  its fundamental tensor is given by
	 \begin{equation*}
	  g_v(u,w) 
	 =\sum_k\varepsilon_k\frac{  \sqrt{L(v)} }{\sqrt{L_k(v)}}h_v^k(u,w) + \sum_{k,l}\frac{\varepsilon_k\varepsilon_l}{\sqrt{L_k(v)}\sqrt{L_l(v)}}g_v^k(v,u)g_v^l(v,w),
	 \end{equation*}
	where $g^k$ and $h^k$ are, respectively, the fundamental tensor and the angular metric of $L_k$, and the  angular metric of $L$ is given by
	 \begin{equation}\label{eq:angularm}
	 h_v(u,w)=\sum_{k=1}^\n\varepsilon_k\frac{\sqrt{L(v)} }{\sqrt{L_k(v)}}h_v^k(u,w),
	 \end{equation}
	 for $v\in A$ and $u,w\in T_{\pi(v)}M$.
\end{lemma}
\begin{proof}
	 Write  $L(v)=(\varepsilon_1\tF_1(v)+\ldots +\varepsilon_\n\tF_\n(v))^2$, where $\tF_k(v)=\sqrt{L_k(v)}$ and  apply Prop. \ref{central} with $\varphi(x_1,\ldots,x_\n)=(\varepsilon_1x_1+\ldots+\varepsilon_\n x_\n)^2$.  Then, clearly,  $\varphi_{,k}=2\varepsilon_k \sqrt{\varphi}$,  $\varphi_{,kl}\equiv 2\varepsilon_k\varepsilon_l$ 
	for $k,l=1,\dots,\n$, and
	\begin{equation*}
	2 g_v(u,w) 
	=\sum_k\frac{2\varepsilon_k  \sqrt{L(v)} }{\tF_k(v)}h_v^k(u,w) + \sum_{k,l}\frac{2\varepsilon_k\varepsilon_l}{ \tF_k(v)\tF_l(v)}g_v^k(v,u)g_v^l(v,w),
	\end{equation*}
	for $v\in A$, as required.  In particular, as $g_v^k(v,v)=\tF_k(v)^2$,
	 \begin{multline*}
	 g_v(v,w)  =  \sum_{k,l}\varepsilon_k\varepsilon_l\frac{g_v^k(v,v)g_v^l(v,w)}{ \tF_k(v)\tF_l(v)}  =\sum_{k,l}\varepsilon_l\frac{\varepsilon_k\tF_k(v)}{ \tF_l(v)}g_v^l(v,w)
	\\
	=  \sum_{l}\varepsilon_l\frac{\sqrt{L(v)}}{ \tF_l(v)}g_v^l(v,w). 
	\end{multline*}
	The angular metric of $L$ is then  
	\begin{multline*}
	h_v(w,w)=g_v(w,w)-\frac{1}{L(v)}g_v(v,w)^2\\=\sum_k\varepsilon_k\frac{\sqrt{L(v)}}{\tF_k(v)}h_v^k(w,w) + \sum_{k,l}\varepsilon_k\varepsilon_l\frac{ g_v^k(v,w)g_v^l(v,w) }{ \tF_k(v)\tF_l(v)}
	- 
	\left( \sum_{l}\varepsilon_l\frac{ g_v^l(v,w)}{ \tF_l(v)}\right)^2.
	\end{multline*}
	As 
	the sum of the last two terms vanishes,  we get \eqref{eq:angularm}.
	\end{proof}
\begin{prop}\label{t_suma}
Let $ L_1,\ldots, L_\n:  \bar A\rightarrow [0,+\infty)   $  be Lorentz-Finsler metrics on $M$.  Then $L=(\sqrt{L_1}+\ldots+\sqrt{L_\n})^2$ is also a Lorentz-Finsler metric. 
\end{prop}
\begin{proof}
 To prove the smoothness of $L$, just notice that 
 $0$ is always a regular value of a Lorentz-Finsler metric 
   (Lemma \ref{l(i)}) 
and, then, 
 the products $\sqrt{L_i L_j}$ are smooth by applying Lemma \ref{lsmooth} to  $M'=TM\setminus\{\mathbf{0}\}$. 
 
 The result is a direct consequence of  Prop.  \ref{indexn-1}  (with $A^*=A$)  if its hypotheses $(i)$ and $(ii)$  hold.  Clearly,
the first one follows from the  expression of $h_v$ in \eqref{eq:angularm}
 ($h_v$ is negative semi-definite and $v$ spans its radical, as these properties hold for all $h_v^k$).  
 For $(ii)$, let $v\in\C$. For any $w$ causal,  $w(L)=(w(L_1)\sqrt{L}/\tF_1+\ldots+w(L_{n_0})\sqrt{L}/\tF_{n_0})>0$, since  each $w(L_i)>0$  (recall   Prop.  \ref{indexn-1})  and $\tF_i/\tF_j> 0$ (by  Lemma~\ref{lsmooth}).  
 Moreover, observing that the cone $\C$ of $L$ coincides with the cone for any $\tF_k$, then
part $(iv)$ of Prop. \ref{propiedades} is applicable to $\C$. Therefore,  $\sigma^w$ is negative semi-definite in $\ker (\omega_v)$ with radical spanned by $v$ and, by formula
 \eqref{secondfundxi} (with $\xi=w$), 
 analogous properties hold for $g_v$. 
\end{proof}
\subsection{General construction
of Finsler spacetimes}
 In \S \ref{newclass}, a simple  new class of Finsler spacetimes was introduced by using a Finsler metric and a one-form. Next, a more general procedure will allow us to  construct  every Lorentz-Finsler spacetime using  Riemannian and  Finsler metrics. 
\begin{thm}\label{t_carac}
Let $\C$ be a cone structure in a manifold $M$ and $A$ its  cone domain  with $\bar A$ its closure in $TM\setminus \mathbf 0$.   A smooth two-homogeneous function   $L:  \bar A   \rightarrow \R$,  $ L\geq 0$, satisfying $L^{-1}(0) = \C$ 
 is a Lorentz-Finsler metric  if and only if there exists a Riemannian metric $\g$  on $M$  and a  conic  Finsler metric
$\cF : A^* \rightarrow \R$ with $ \bar A  \subset A^*  $   such that 
\begin{equation}\label{ClassEx}
 L(v)= \g(v,v) - \cF(v)^2 ,  \,  \qquad \qquad \forall v\in \bar A,
\end{equation}
   and the following  properties  hold for any 
   $v\in  \bar A $:

(i)  whenever $v\in A$ (i.e., $L(v)>0$), 
\begin{equation}\label{fundtensEx}
\g(w,w)-\hat g_v(w,w)  -  \frac{1}{L(v)}\hat g_v(v,w)^2<0, \qquad \qquad  \forall
w\in \left< v\right>^{\perp_{\g}}\setminus\{0\}, \, \end{equation}
 where, as natural $\left< v\right>^{\perp_{\g}}:=\{w\in TM: \g(w,v)=0\}$,

(ii)  whenever $v\in \C$ (i.e., $L(v) =0$), 
\begin{equation}\label{fundtensEx2}
\g(w,w)-\hat g_v(w,w)<0 , \qquad \qquad \forall w\in \left< v\right>^{\perp_{\g}}\cap \left< v\right>^{\perp_{\hat g_v}} \setminus \{0\} 
\end{equation}
 and the indicatrices of $\g$ and $\hat F$ intersect transversely at $v$,  namely,  $\left< v\right>^{\perp_{\g}}\not= \left< v\right>^{\perp_{\hat g_v}} $. 
\end{thm}
\begin{proof}
 Assume first that $L$ is a Lorentz-Finsler metric.  Then if $\g$ is a Riemannian metric, we can define an auxiliary pseudo-Finsler metric  $\hat{L}$  given by  $\hat{L}(v)=\g(v,v)-L(v)$  whose fundamental tensor satisfies
\[\hat{g}_v(u,w)=  \g(u,w) - g_v(u,w),  \]
being  $g$ the fundamental tensor of $L$.  At each $p\in M$,  the set of directions in $\bar A \cap T_pM$ is compact; so, up to a conformal re-scaling in the choice of $\g$, 
we can assume that $\hat g_v$ is positive definite at all $v\in  \bar A $, obtaining a conic  Finsler metric $\cF=\sqrt{\hat{L}}$  defined in some $A^* \supset  \bar A  $ (where $L$ is also extendible). 

 So, it is enough to check that a pseudo-Finsler metric as in \eqref{ClassEx}  is non-degenerate with index $n-1$  (that is, the conditions {\em (i)} and {\em (ii)} in Prop.~\ref{indexn-1} hold)  if and only if 
the conditions $(i)$ and $(ii)$  above hold.  Recall that the angular metric $h$ of $L$ on $A$ is determined by 
\[
 h_v(w,w)= \g (w,w)-\hat g_v (w,w) -\frac{1}{L(v)}{\left(\g (v,w)-\hat g_v (v,w)\right)^2}.
\]
 Now,  $h_v$ is  negative  semi-definite with radical spanned by $v$ (i.e. {\em (i) } in Prop. \ref{indexn-1} holds) 
if and only if it is negative in a transverse hyperplane to $v$.  Choosing   such a  hyperplane as $\left< v\right>^{\perp_{\g}}$, 
 this is equivalent to
 \eqref{fundtensEx}. 
About {\em (ii)},  the transversality of the indicatrices at $v$ is equivalent to saying that 
$\omega_v :=g_v(v,\cdot)= g_R(v,\cdot)-\hat g_v(v,\cdot)\not\equiv 0$ (indeed, $\omega_v(v)=0$; so, $\omega_v\equiv 0$ is equivalent to $\left< v\right>^{\perp_{\g}}= \left< v\right>^{\perp_{\hat g_v}}$). In this case,
$\left< v\right>^{\perp_{\g}}\cap \left< v\right>^{\perp_{\hat g_v}} $ has dimension $n-2$ and it is contained in $ \left< v\right>^{\perp_{g_v}}=\ker \omega_v$.  Therefore,  
{\em (ii) } above becomes  equivalent to $(ii)$ in  Prop.  \ref{indexn-1}.
\end{proof}

\begin{rem}\label{r_fibrado} This theorem can be used to characterize not only  the Lorentz-Finsler metrics on a manifold $M$ but also on a vector bundle. In particular,  consider the bundle $\pi^*_M(S) \rightarrow S$ used for the characterization of stationary spacetimes in Section \ref{ex_stationary}. Any Lorentz-Finsler metric $L^S$ in this bundle can also be written as a difference type $g_R-\hat F^2$. However, in the stationary setting, we were interested in the case that the Killing vector field $K$ (which could be  naturally identified with $\partial_t$ on $\R\times S$)  was timelike, that is,  $L^S(K)>0$. This condition can be ensured just by imposing that the direction of $K$ lies in the cone domain $A$ determined by the intersection of the indicatrices of $g_R$ and $\cF$.
\end{rem}
In spite of the generality of
Th. \ref{t_carac}, its application to construct Lorentz-Finsler metrics is not so straightforward as in
Th. \ref{t_examp}. Indeed, one has to check not only the conditions $(i)$ and $(ii)$ but also that  the ``appropriate'' cone domain $A$ has been chosen, as the following example shows.

\begin{exe}
Choose $g_R=2dx^2+2dy^2+dz^2$ and $\cF=\sqrt{dx^2+dy^2+2dz^2}$ in $\R^3$. Then $g_R-\cF^2=dx^2+dy^2-dz^2$	does not satisfy \eqref{fundtensEx} and \eqref{fundtensEx2} in any point of the region $A=\{v\in \R^3:g_R(v,v)-\cF(v)^2>0\}$. 
	
%
\end{exe}

\section{Lorentz-Finsler metrics associated with a cone structure}\label{s5}

Next our aim is to prove a general smoothing procedure for Lorentz-Finsler metrics which, in particular, will show that any cone structure $\C$ can be regarded as the cone structure associated with a (smooth) Lorentz-Finsler metric. 

\subsection{Non-smooth Lorentz-Finsler  $L$  associated with a cone triple}\label{s51} Recall that any cone structure $\C$ 
was determined by some cone triple $(\Omega, T, F)$ (Th. \ref{p_transversality}), which also yielded the decomposition \eqref{e_OmegaDecomposition} of $TM$. A first result of compatibility with Lorentz-Finsler metrics is the following.

\begin{prop}\label{l_G} For any cone triple $(\Omega, T, F)$ of a cone structure $\C$, the continuous two-homogeneous function $G: TM\rightarrow \R$, 
\begin{equation}\label{e_contLF}
G(tT_p+w_p)= t^2 - F(w_p)^2, \qquad \forall t\in\R, \; \forall w_p \in \ker (\Omega _p), \, \forall p \in M,
\end{equation}  is smooth everywhere but on\footnote{ From the proof and \cite[Th. 4.1]{Warner65}, it follows that $G$ will be smooth on span$(T)$  if and only if $F$ comes from a Riemannian metric.} span$(T)$. Moreover, whenever it is smooth,  its fundamental tensor $g$ (computed as in \eqref{fundtens}) is non-degenerate with index $n-1$. 
Such a $G$ will be called the {\em continuous Lorentz-Finsler metric associated with} $(\Omega, T, F)$.
\end{prop}

\begin{proof}
The smoothness of $G$  follows directly by taking local  fibered coordinates on $TM$ using a reference frame $(T=X_1, X_2, \dots , X_n)$ where $\ker (\Omega)= $ span$\{X_2, \dots , X_n\}$ (to construct this, choose a coordinate frame 
$(T=\partial_1, \partial_2, \dots , \partial_n)$ and project on $\ker(\Omega)$ in the direction of $T$). Then, using \eqref{fundtens} the fundamental tensor $g$ of $G$ and $\hat{g}$  of $F$ are related by
$$
g=\Omega^2 - \pi^* \hat{g},
$$
where $\pi$ is the projection onto $\ker (\Omega)$  (as in \eqref{e_OmegaDecomposition}). From the last identity, it follows straightforwardly  that the index of $g$ is $n-1$ as required. 
\end{proof}

\begin{rem}  The  indicatrix associated with  the triple $(\Omega, T, F)$ is then:
\begin{equation}
\label{e_indicatrixCONT}
\Sigma := G^{-1} (1)\cap \Omega^{-1}((0,\infty)).
\end{equation}
Clearly, $\Sigma_p:= \Sigma \cap T_pM$ will be a convex hypersurface and it is smooth and strongly convex  with respect to the position vector  everywhere except in $T_p$. 
However, $G$ provides a second cone (and, thus,  another Lorentz-Finsler metric). Indeed,
$G(tT_p+w_p)=G((-t)T_p+w_p)$, so, we will have  a ``reflected'' cone structure with indicatrix
$$
\Sigma^- =\{-t T_p+w_p: t T_p+w_p \in \Sigma\}.
$$
\end{rem}

  \smallskip

Next, our aim is to smooth $G$ around $T$. With this purpose,  $\Sigma$ will be smoothed by constructing a new hypersurface $\tilde\Sigma$ s.t.: 

(i) it is strongly convex, 

(ii) pointwise  $\tilde\Sigma_p= \Sigma_p$ outside a relatively compact neighborhood of  $T_p$. 

\smallskip

Once constructed $\tilde \Sigma$ and the reflected one $\tilde \Sigma^-$,  the required smooth Lorentz-Finsler $\tilde G$ will be determined by imposing:  

(a) pointwise  $\tilde G= G$ outside the radial directions perturbed of $\Sigma$, and 

(b) $\tilde \Sigma$ and $\tilde \Sigma^-$ are, resp.,  the future and past indicatrices of $\tilde G$.

\begin{rem}\label{e_5.3}
(1) We will focus in this concrete problem of smoothness, especially adapted to cone structures. However, the smoothing procedure is very general and could be applied to  any other    continuous  Lorentz-Finsler metric whose indicatrix is  convex but non-smooth  in a (pointwise) compact subset.  In fact, it can be applied to any Finsler spacetime as defined in \cite{AaJa16}.  

(2) A different problem would happen for non-smooth cone structures. However, its description by means of a triple $(\Omega,T,F)$ would reduce this question  to smoothen (some of) these elements.
\end{rem}


\subsection{The smoothing procedure of indicatrices} A smooth function $f$ defined on $\R^m$ will be called strongly convex when its Hessian, Hess$(f)$, is definite positive; thus, its graph will be a strongly convex hypersurface.  In the following, if $D$ denotes a disk of radius $r$, we will denote by $D/2$ and $D/4$ the disks with the same center but radius $r/2$ and $r/4$, respectively. 

\begin{lemma}\label{l_PRINC} (A  strongly convex approximation  for  a convex function). Let $t_0: \R^{n-1}\rightarrow \R$ be a continuous convex function which is smooth and  strongly convex  everywhere but in $0$  and such that there exists a neighborhood of $0$ where the $\hess( t_0)$ is lower bounded by a positive constant except in zero.  Let $D$ be a disk (a closed ball centred at the origin of radius $r>0$). Then, for any $\epsilon>0$ there exists a strongly convex function $\tilde t_0$ defined on all $\R^{n-1}$ such that $t_0=\tilde t_0$ away from $D/2$ and $|\tilde t_0-t_0|<\epsilon$ everywhere.
\end{lemma}
\begin{proof} Let $\{\mu_0 , \mu_1\}$ be a partition of  unity subordinated to the covering 
$\{(D/2), \R^{n-1}\setminus (D/4)\}$. 
Let $\hat t_0$ be a strongly convex function such that  $|\hat t_0-t_0|<\epsilon$ on $D$ and, even more,
\begin{equation}\label{e_BOUNDS}
 |\hat t_0-t_0|, 
|\hbox{grad} (\hat t_0-t_0)|, 
|\hess (\hat t_0-t_0)| <\hat\epsilon , 
\end{equation}
 ($|\cdot |$ denotes the usual norm of the corresponding element, regarding it as included in $\R,\R^{n-1}$ or $\R^{(n-1)^2}$, resp.) on the closure of $(D/2)\setminus(D/4)$  for some $\hat\epsilon>0$ to be specified below 
(such bounds  can be obtained for arbitrarily small $\hat\epsilon>0$  by the standard theory of convex functions)\footnote{ On the one hand, the strong convexity of $t_0$ and the compactness of
the boundary of $D/2$ allows one to find a function $f$ as in   \cite[Th. 2.1]{Gh}, which is convex and $(\hat\epsilon/2)$-close to $t_0$  everywhere, agrees with $t_0$ outside $D/2$ and has first and second derivatives $(\hat\epsilon/2)$-close to $t_0$ on   
 $(D/2)\setminus(D/4)$ (for the latter,  recall \cite[formula (2.3)]{Gh}). On the other, the lower boundedness of $\hess(t_0)$ ensures that $t_0$ is strongly convex in the sense of \cite[Def. 1]{Azagra} and allows  one  to find a strongly convex function $g$ which is $(\hat\epsilon/2)$-close to $t_0$ 
 \cite[Cor. 1]{Azagra}. So, for small $\eta>0$, the linear combination $\eta g + (1-\eta)f$  makes the job.}.
 Let, 
$$
\tilde t_0= t_0 + \mu_0(\hat t_0-t_0),
$$
that is clearly smooth and equal to $t_0$ away from $D/2$. This function 
will become strongly convex, as required, just by making $\hat \epsilon$ smooth enough so that the Hessian of the last term is smaller on $D\setminus (D/4)$ than the Hess $t_0$. Concretely,
\begin{multline}
\label{e_LAST}
\hess \tilde t_0= \hess t_0 + (\tilde t_0-t_0) \hess \mu_0 + \hbox{grad} (\tilde t_0-t_0) \hbox{grad} \mu_0\\+ \hbox{grad} \mu_0\hbox{grad} (\tilde t_0-t_0) + \mu_0 \hess (\tilde t_0-t_0).
\end{multline}
As there are  $\nu, C>0$ such that 
$\hess t_0> \nu$ and $\mu_0, | \hbox{grad} \mu_0 | , | \hess \mu_0| < C$ in the closure of $(D/2)\setminus(D/4)$, the choice $\hat\epsilon<\nu/(4C)$ suffices.
\end{proof}

\begin{rem}\label{r_LAST}
  The previous proof  can be extended  directly  to other cases discussed in Appendix~\ref{s_a1}. However, the next argument by D. Azagra provides a much more direct proof. Let $D$ be the unit disk with no loss of generality and $\xi$ be a subgradient of $t_0$ at $0$. As $\textrm{Hess}(t_0)$ is bounded from below by a constant $\delta>0$, 
$$
t_0(x)\geq t_0(0)+\langle \xi, x\rangle +\frac{\delta}{2}\|x\|^2 \qquad \forall x\in\R^n.
$$
So, outside $D/2$,  where $D$ can be regarded as unit disk with no loss of generality,  we have
$
t_0(x)\geq 
t_0(0)+\langle \xi, x\rangle +\frac{\delta}{4}\|x\|^2 +\frac{\delta}{16},
$
while in a sufficiently small neighborhood of $0$, we have
$
t_0(0)+\langle \xi, x\rangle +\frac{\delta}{4}\|x\|^2 + \frac{\delta}{32}> t_0(x) +\frac{\delta}{64}.
$
Let us define
$$
\widetilde{\varepsilon}=\min\left\{\frac{\varepsilon}{2}, \frac{\delta}{64}\right\},
\quad  
\widetilde{t}_0(x)=M_{\widetilde{\varepsilon}}\left( t_0(x), \,\, t_0(0)+\langle \xi, x\rangle +\frac{\delta}{4}\|x\|^2 + \frac{\delta}{32}\right),
$$
where $M_{\widetilde{\varepsilon}}$ is the {\em smooth maximum} of \cite[Prop. 2]{Azagra}. Then we have $\widetilde{t}_0=t_0$ off of $D/2$,  $|\widetilde{t}_0 -t_0|\leq \widetilde{\varepsilon}<\varepsilon$ everywhere, and $\widetilde{t}_0$ is a strongly convex function  (indeed,  so it is   at $x=0$  because  $M_{\tilde \varepsilon}$ is equal to the second function  around 0, \cite[part (3) of Prop. 2]{Azagra}, and away from zero  by  \cite[part (9) of Prop. 2]{Azagra}). 

\end{rem}

Next, this lemma  will be applied pointwise to $\Sigma$  in \eqref{e_indicatrixCONT}, regarding each $\Sigma_p$ as  the graph of a convex function  on  $\ker$ $\Omega_p$. 

\begin{thm} \label{t_principal} Let  $(\Omega, T, F)$ be a cone triple  
on $M$ with cone $\C$  and let $G$ be its associated continuous Lorentz-Finsler metric \eqref{e_contLF} with indicatrix $\Sigma$. 

Let  $\mathcal{U}$ be any  neighborhood  of  the section $T$ regarded as a submanifold of 
 $TM$,  which will be assumed (without loss of generality) with the closure of  
$\mathcal{U}\cap T_p M$ compact and included in the cone domain $A$, for all $p\in M$. 

Then, there exists a smooth  hypersurface $\tilde \Sigma \subset TM$ 
satisfying: 

(a) $\tilde \Sigma = \Sigma$ in $TM$ away from 
$\mathcal{U}$. 

(b) Each $\tilde \Sigma_p=\tilde \Sigma \cap T_pM$ is transverse to all the radial  directions in $A$, and $\tilde\Sigma_p$ is strongly convex  (with respect to the position vector) everywhere.
\end{thm}

\begin{proof}
Consider the function $\tau: \ker (\Omega)\rightarrow \R$ determined univocally by 
$$\tau(u_p)T_p+u_p \in \Sigma_p , \qquad \forall u_p\in \ker(\Omega), \quad \forall p\in M.$$
For each $p$,  let $\tau_p$ be its restriction to $\ker(\Omega_p)$,  and introduce local fibered coordinates for $\ker (\Omega)$
by taking a small open coordinate chart $(U,\phi)$ centered at $p$ and choosing a basis of $n-1$ vector fields that expand  $\ker(\Omega)$ on $U$. In such coordinates,  each function $\tau_q$, $q\in U$, is written as a function 
 $t_{x}:\R^{n-1}\rightarrow \R$ labelled with $x=\phi(q)$; in particular, $\tau_p=t_0$.
Varying $x\in \phi(U)$ we have then a function: 
 $$
t: \phi(U)\times \R^{n-1}\rightarrow \R, \qquad \qquad (x,y)\mapsto t_x(y),
$$   
which is smooth in $(x,y)$ away from $y\equiv 0$ because of the properties of smoothness and continuity of $\Sigma$ and the transversality of every $\Sigma_p$ to each line $\{u_p+ \lambda T_p: \lambda\in\R\}$, $u_p\in \ker(\Omega_p)$.  Moreover,  as $t_x(v)=\sqrt{1+F(v)^2}$,
\[\hess_v(t_x)(u,w)=\frac{1}{\sqrt{1+F(v)^2}}\left(g_v(u,w)-\frac{1}{1+F(v)^2}g_v(v,u)g_v(v,w)\right),\]
which is lower bounded in any bounded  neighborhood of zero (away from zero) because $\hess_{\lambda v}(t_x)(u,u)=\frac{1}{1+\lambda^2F(v)^2}g_v(u,u)$ if $u$ is $g_v$-orthogonal to $v$ and $\hess_{\lambda v}(t_x)(v,v)=\frac{F(v)^2}{1+\lambda^2 F(v)^2}$.

Clearly, $t_0$ lies under the hypotheses of  Lemma \ref{l_PRINC}, and we can take $\hat t_0$, $\mu_0$ and $\mu_1$ as in its proof for some small disk $D$ such that 
$T_q+u_q\in \mathcal{U}$ for all $u_q\in \ker (\Omega)$ with coordinates in 
$\phi(U)\times (D/2)$.
  Now, regard $\hat t_0$, $\mu_0$ and $\mu_1$ as functions on all $\phi(U)\times \R^{n-1}$ just making them independent of the variable $x\in \phi(U)$. Choosing a smaller $U$ if necessary, the continuity of $t$ and its derivatives ensure that the bounds \eqref{e_BOUNDS}   in the closure of $(D/2)\setminus(D/4)$ hold not only for $\hat t_0-t_0$ but also for  $\hat t_0-t_x$ for all $x\in \phi(U)$
and  some convenient $\hat \epsilon$. 
Concretely, $\hat \epsilon$ is chosen 
so that the bounds below \eqref{e_LAST} hold with $\hess t_x>\nu$ for all $x\in \phi(U)$. Therefore,
$$
\widetilde{t}(x,y):= t_x(y)+\mu_0(y)(\hat t_0(y)-t_x(y)) \qquad \qquad \forall (x,y)\in \phi(U) \times \R^{n-1}
$$
is smooth, strongly convex and it satisfies 
$
\widetilde{t}=t$ outside $ \phi(U) \times (D/2)$. 

Now, consider the function $\tilde \tau:\ker (\Omega ) \cap TU\rightarrow \R$ whose expression in coordinates is  $\tilde t$,  and define its graph as follows:
$$\hbox{Graph}(\tilde \tau)=\{\tilde \tau(u_p)T_p+u_p : u_p\in \ker(\Omega)\cap TU, \,  p\in U\}.$$
Clearly, $\hbox{Graph}(\tilde \tau)$ is a hypersurface which fulfills all the required properties for $\tilde \Sigma$ except that it is defined only on $TU$. In order to obtain an appropriate function $\tilde \tau^*$ on all $\ker (\Omega)$ function, consider for each $p\in M$ the constructed function $\tilde\tau\equiv \tilde \tau^p$ and neighborhood $U\equiv U^p$, 
and take a subordinated partition of  unity $\{\rho_i:$ supp$(\rho_i) \subset U^{p_i}, i\in\N\}$.  Then, the  pointwise linear combination of strongly  convex functions, 
$$
\tilde \tau^*: \hbox{ $\ker$ }\Omega \rightarrow \R, \qquad \qquad \tilde \tau^* = \sum_{i\in\N} (\rho_i\circ \pi_M) \cdot \tilde\tau^{p_i}
$$
(with $\pi_M: TM\rightarrow M$ the natural projection) yields  the required $\tilde \Sigma$. 
\end{proof}

\begin{rem} \label{r_principal}  (1)  A natural way to choose such a small neighborhood $\mathcal U$ of the section $T$ is  as follows.   Given any (continuous) function $\epsilon: M\rightarrow \R$ with $0<\epsilon<1$, one can take: 
$$\mathcal{U}=\{\lambda(p)  T_p + w_p: 1-\epsilon(p)<\lambda(p)<1+\epsilon(p) ,F(w_p)<\epsilon(p), 
p\in M\}.$$  

 (2) 
In particular, choose $\epsilon \equiv 1/2$. Once $\tilde \Sigma$ has been obtained, a smooth function $\lambda$ on $M$, $1/2<\lambda<3/2$, is obtained by imposing $\lambda(p)T_p \in \tilde \Sigma$ on all $M$. The triple 
$(\Omega/\lambda,\lambda T,  F/\lambda)$ 
is also associated with the cone $\C$.  However,  the continuous Lorentz-Finsler metric associated with this triple is different to both, $\Sigma$ and $\tilde \Sigma$, in general.

 (3) Trivially, the smoothing procedure can be carried out in a way independent of the Killing vector $K$ for a (continuous) standard   stationary spacetime, obtaining so new examples of smooth stationary spacetimes starting at non-smooth ones (for example, starting just at a product $(\R\times S, G\equiv -dt^2+F^2)$, where $F$ is a classical Finsler metric on $S$).
\end{rem}


 As a direct consequence, we obtain: 
\begin{cor}\label{t_PRINNCIPAL}  Any cone structure $\C$ is the cone structure of a (smooth) Lorentz-Finsler metric  $L$  defined on all $TM$. 
\end{cor}
\begin{proof}
Consider any cone triple $(\Omega,T,F)$ associated with $\C$ and the corresponding continuous Lorentz-Finsler metric $G$ in Prop. \ref{l_G}. The required metric $F$ is just the metric $\tilde G$ obtained by smoothing $G$ (using Th. \ref{t_principal}) as explained  in the paragraph before  Rem. \ref{e_5.3}. 
\end{proof}
  A proof of the last  corollary  with a different approach can be found in \cite[Prop. 13]{Min17}. 
\begin{rem}\label{rPRINCIPAL} This result and Th. \ref{t_anisotropic} can be summarized as follows: {\em each cone $\C$ determines univocally a (non-empty) class of anisotropically conformal Lorentz-Finsler metrics}.
\end{rem}

\section{Cone geodesics and applications}\label{s6}

 As seen in  Subsection  \ref{s2_con_est_causality}, there is an obvious way to extend  the causality of relativistic spacetimes to any  cone   structure $\C$ and, thus, to any Lorentz-Finsler metric.
However, the fact that  any such a $\C$ can be regarded as associated with a Lorentz-Finsler metric $L$  has a double interest now. On the one hand, this allows one to identify the causal elements of $L$, including notably its lightlike pregeodesics, 
as elements inherent to any cone structure. On the other  hand,  the existence of $L$ yields an additional analytical tool to understand such elements. 


\subsection{Summary on maximizing Lorentz-Finsler causal geodesics}\label{s_61}
As a preliminary question, notice that  causal geodesics in a Finsler spacetime have properties of maximization among (piecewise smooth) causal curves completely analogous to those of classical spacetimes. This has already been   pointed out by several authors  \cite{AaJa16,Min15a}  and will be summarized here following \cite{ON}.
 Along this subsection, let $L: \bar A\rightarrow  [0,+\infty) $ be
 a Lorentz-Finsler metric, extended to some conic domain $A^*\supset \bar A$ (to avoid issues on differentiability).  Consider
its Chern connection on the whole $A^*$ and, then,  its geodesics 
 and  exponential map 
$\exp: \mathcal{D} \rightarrow M$,  
where  $\mathcal{D}\subset   A^*$ is  maximal and starshaped with $\mathcal{D}\setminus \mathbf{0}$  open. The  smooth variation of the solutions to the geodesic equation with the initial conditions implies the smoothness of $\exp$ 
 away from  $\mathbf{0}$. Let $\mathcal{D}_p:=\mathcal{D}\cap T_pM$ and $\exp_p = \exp|_{\mathcal{D}p}$, and consider the following two lemmas (the first one is a basic  result, see \cite[Prop. 6.5]{AaJa16}):

\begin{lemma}\label{lem:closetimelike}
Let $\alpha:[a,b]\rightarrow M$ be a causal curve which is not a pregeodesic. Then there exist   timelike curves from $\alpha(a)$ to $\alpha(b)$ arbitrarily close to $\alpha$. 
\end{lemma}


\begin{lemma}\label{lem:existD}
 For any $p\in M$ there   exists 
a neighborhood $\tilde D$ of zero in $T_pM$ such that $D:=\tilde D \cap A^*\subset T_pM$  is starshaped and connected,
and 
 $\exp_p$ is defined in $D$, being
 $\exp_p:D\rightarrow \exp_p(D)$
a diffeomorphism. 
\end{lemma}
\begin{proof}
 This is the  analog to normal neighborhoods and it can be obtained from  a local extension of
the Chern connection  beyond $A^*$, regarding it as  an anisotropic connection  (see  \cite[Chapter 7]{Sh01} or \cite{Aniso16}).  Taking a  coordinate neighborhood $U$  around $p\in M$, 
the Chern connection is determined by the Christoffel symbols $\Gamma_{ij}^k:A^*\cap TU\rightarrow \R$ 
(see \cite[\S 2.4 and 2.6]{Aniso16}). 
As they are  0-homogeneous, they can be regarded as functions with domain in the unit bundle $S_RM$ (for some auxiliary Riemannian metric).  Being $S_RM\cap TU \cap \bar A $  a 
closed subset of $S_R M \cap TU$, all $\Gamma_{ij}^k$  can be extended  to functions  $\tilde \Gamma_{ij}^k:  TU \setminus\mathbf{0} \rightarrow \R  $
such that $\tilde \Gamma_{ij}^k=\Gamma_{ij}^k$ on $TU \cap \bar A$, $\tilde \Gamma_{ij}^k=0$ on $TU \setminus A^*$  and they are homogeneous of degree 0 everywhere. 
This anisotropic  connection has a natural exponential map, $\tilde\exp_p$,  which is $C^1$  and whose differential in $p$ is the identity\footnote{The fact that this exponential map is $C^1$ and it admits convex neighborhoods was 
proved by Whitehead \cite{whitehead} 
and it can also be proved 
as in \cite[\S 5.3]{BaChSh00}, where  the exponential map of a Finsler metric is considered.}.
Applying the inverse function theorem we obtain a  starshaped  neighborhood $\tilde D$ of $0$ in $T_pM$ such that $\tilde\exp_p$ is a diffeomorphism in $\tilde\exp_p(\tilde D)$,  thus, satisfying the required properties. 
\end{proof}
\begin{rem}\label{globalCon}
 In the previous proof, the local extension of the Chern  connection was enough to obtain normal neighborhoods. However, this connection can be extended globally to the  tangent bundle  as follows.  Consider for each $p$ the neighborhood  $U_p$  and the anisotropic connection $ \tilde \nabla^{(p)}  $  on $U_p$, defined for all the non-zero directions of tangent bundle $TU_p$, and which extends the Chern connection of $L$ as  in the above lemma. Let $\{\mu_i\}_{i\in \N}$ be a subordinate partition of unity, with each supp$(\mu_i)$ included in some $U_{p_i}$. 
So, the required connection is just 
$$\tilde \nabla:=  \sum_i (\mu_i\circ \pi)  \tilde \nabla^{(p_i)}.$$
 Observe that, even though the anisotropic connections do not constitute a $C^\infty(M)$-module, 
the expression above does define an anisotropic connection; in particular, the Leibniz rule is satisfied  because\footnote{ The space of all the anisotropic connections (as well as the space of all the linear connections) is naturally an affine space, and the condition $\sum_i \mu_i=1$ can be interpreted as the natural restriction of local  barycentric coordinates. }  $\sum_i (\mu_i\circ \pi)=1$.

\end{rem}
\begin{lemma}\label{lem:timelikeexp}
If $\beta:[a,b]\rightarrow T_pM$ is a (piecewise smooth) curve such that its image   lies  in a domain $D  \subset  T_pM$ of $\exp_p$  as in Lemma \ref{lem:existD},  and $\alpha=\exp_p\circ \beta$  is a causal curve, then $\beta$ remains inside the  causal cone of $\C_p$, and if $\alpha$ is timelike, without touching $\C_p$. 
\end{lemma}
\begin{proof}
When $\alpha$ is timelike, it follows the same lines as \cite[Lemma 5.33]{ON}. Namely, use  the Gauss Lemma (see \cite[Rem. 3.20]{JavSan11}) 
in order to show that $L_p(\beta)$ is always positive. Indeed, it is positive at least at a small interval $[a,a+\varepsilon]$ because $\dot\beta(a)\equiv\dot\alpha(a)$  is a timelike vector, and its derivative is  positive  by the Gauss Lemma, as 
\[\frac{d}{dt}L_p(\beta)=2g_{\beta}(\beta,\dot\beta)=   2g_{d(\exp_p)_\beta(\beta)}(d(\exp_p)_\beta(\beta),\dot\alpha) >0\]  (the latter because  both $d\exp_p(\beta)$ and $\dot\alpha$ are  timelike, recall  \cite[Prop. 2.4]{AaJa16}). This guarantees that   $\beta$ lies  in the cone up to the first  break; however,  here the Gauss Lemma can be used again to guarantee that  $\beta $ remains 
in the timelike cone. Assume now that $\alpha$ is causal. If $\alpha$ is not a lightlike pregeodesic, then there exists a timelike curve very close to $\alpha$  (this is a consequence of 
Lemma~\ref{lem:closetimelike}; however, to fix the endpoint is not required here),  which reduces the proof to the first case in that $\alpha$ is timelike. 
\end{proof}
\begin{prop}\label{minimizeconic} 
Let $(M,F)$ be a Lorentz-Finsler metric, $p\in M$  and $D$ an open subset of $T_pM$ as in Lemma \ref{lem:existD}.  Then, for any $q\in \exp_p(D)$ the
radial geodesic from $p$ to $q$ is, up to reparametrizations, the
unique maximizer of the Finslerian separation among the causal curves
contained in $\exp_p(D)$.
\end{prop}
\begin{proof} 
 Assume first that the radial geodesic $\sigma:[0,b]\rightarrow \exp_p(D)$ is lightlike. If  there is any other causal curve 
from $p$ to $q$ which is not a lightlike pregeodesic, then by Lemma \ref{lem:closetimelike}, there is a timelike curve from $p$ to $q$ and by Lemma \ref{lem:timelikeexp}, $\exp_p^{-1}(q)$ lies in the timelike cone. Therefore,  there is a timelike radial geodesic from $p$ to $q$, in contradiction with the fact that $\exp_p$ is  a diffeomorphism on $D$. It follows that $\sigma$ is the only causal curve from $p$ to $q$ in $\exp_p(D)$ and it is also the unique maximizer. 

 Assume now that the radial geodesic $\sigma:[0,b]\rightarrow \exp_p(D)$ is timelike. If $\alpha:[0,b]\rightarrow M$ is another causal curve from $p$ to $q$ in $\exp_p(D)$ which is not  a reparametrization of $\sigma$, let $c\in [0,b)$ be the biggest instant such that $\alpha|_{[0,c]}$ is a lightlike pregeodesic. It follows from Lemmas \ref{lem:closetimelike} and \ref{lem:timelikeexp} that $\exp_p^{-1}(\alpha|_{(c,b]})$ is 
 contained in the timelike cone of $T_pM$. Let $\tilde P$ be the position vector in $T_pM$ and $r=\sqrt{L}$ defined in the timelike cone of $T_pM$.  So, $\tilde{\mathfrak  u} :=  \tilde P/r$ is a unit timelike vector of $T_pM$ for every $\tilde P$ in the timelike cone, and  putting  $  {\mathfrak  u}  :=  d(\exp_p)_{\tilde {\mathfrak  u}}(\tilde{\mathfrak  u})$,  by the 
 fundamental inequality (see Appendix~ \ref{s_a2}) $F(\dot\alpha(t))\leq g_{ {\mathfrak  u}}({\mathfrak  u},\dot\alpha(t))$ for every $t\in (c,b)$ (recall that $F=\sqrt{L}$). Moreover,  let us define $\beta=\exp_p^{-1}(\alpha)$ and observe that $d_vr=\frac{1}{2r} d_vL=\frac{1}{r} g_v(v,\cdot)$ and then $g_{\tilde {\mathfrak  u}}({\tilde {\mathfrak  u}},\dot\beta(t))=\frac{d(r\circ\beta)}{dt}$. It follows that
 \begin{multline*}
 \int_0^b F(\dot\alpha)dt=\int_c^b F(\dot\alpha)dt\leq \int_c^b g_{\mathfrak  u}({\mathfrak  u},\dot\alpha(t))dt= \int_c^b g_{\tilde {\mathfrak  u}}({\tilde {\mathfrak  u}},\dot\beta(t))dt\\=\int_c^b \frac{d(r\circ\beta)}{dt} dt=  r(\exp^{-1}(q)),
 \end{multline*}
 where we have used the Gauss Lemma (see \cite[Rem. 3.20]{JavSan11}) in the second equality,  and the equality holds if  and only if $\alpha$ is a reparametrization of $\sigma$. Observe that if $\alpha|_{[c,b]}$ is a reparametrization of $\sigma$, then $\exp_p^{-1}(\sigma|_{[c,b]})$ cannot touch the lightlike cone at $c$ away from $0$ and it follows that $c=0$. 
 \end{proof}

\notshow\input{conformes}\notshowend


\subsection{Cone geodesics vs lightlike pregeodesics}\label{s_62}


Recall that cone geodesics were defined as locally horismotic curves inherent to any cone structure (Def. \ref{def_cone_g}).
\begin{thm} \label{t_CONE} Let $\C$ be a cone structure 
and   $\gamma: I\subset \R\rightarrow M$ a 
curve. The following properties are equivalent:

(i) $\gamma$ is a  cone geodesic of $\C$ 

(ii) $\gamma$ is a lightlike pregeodesic for one (and, then, for all) Lorentz-Finsler metric $L$  with cone $\C$. 

\smallskip In particular, all anisotropically equivalent Lorentz-Finsler metrics have the same lightlike pregeodesics.

\end{thm}
\begin{proof} 
{\em (ii)} $\Rightarrow$ {\em (i)}. Straightforward from the maximizing properties of lightlike pregeodesics stated in Prop. \ref{minimizeconic}. 

{\em (i)} $\Rightarrow$ {\em (ii)}. Cor. \ref{t_PRINNCIPAL} ensures that there exists at least one $L$; then, the uniqueness of the maximizing properties in Prop. \ref{minimizeconic} concludes.
\end{proof}
\begin{rem} \label{r_CONE} The existence of the metric $L$ compatible with $\C$ allows us to introduce  an exponential map and recover all the classical Causality Theory of spacetimes. This includes the so-called {\em time separation} or {\em Lorentzian distance} for $L$, which is not conformally invariant. 

However, just the exponential for lightlike geodesics suffices  for  conformally invariant ones. So, from a practical viewpoint,  cone triples $(\Omega, T, F)$ associated with $\C$ may yield a notable simplification.  Remarkably,  the  continuous Lorentz-Finsler metric $G$ associated with
$(\Omega, T, F)$ is a very simple metric  and it suffices for the computation 
of cone geodesics (even though $G$  had  the drawback of being non-smooth in the 
$T$-direction,  see Prop.~\ref{l_G},  such a direction is timelike; so,  its lightlike 
geodesics are determined by equations which 
depend on smooth elements at the cone). 

\end{rem}
 Finally, let us consider other natural notions inherent to $\C$. 
\begin{defi}
 Given a cone structure $(M,\C)$,  a submanifold $P\subset M$ is {\em orthogonal to a cone geodesic} $\gamma:[a,b]\rightarrow M$ in $\gamma(a)\in P$ if $T_{\gamma(a)}P\subset T_{\dot\gamma(a)}\C_{\gamma(a)}$.
 \end{defi}
 Clearly,  this definition coincides with the concept of orthogonality provided by any Lorentz-Finsler metric compatible with $(M,\C)$. 
 Then, the notions of conjugate and focal points can also be extended to $\C$. Indeed, the   invariance of lightlike pregeodesics by means of anisotropically equivalent transformations can also be proven by a direct study of geodesic equations \cite{JavSoa18}. Furthermore, lightlike geodesics cannot be maximizers  (among close causal curves)  of the Finslerian separation after the first conjugate point (see \cite[Th. 6.9]{AaJa16}),  which implies that the first conjugate point must be invariant by anisotropic transformations. In fact, it can be shown that all the conjugate and focal points are invariant by anisotropic transformations (see \cite{JavSoa18} and \cite[Th. 2.36]{MinSan}  for the Lorentzian case).  Summing up, one has the following consistent notion. 
\begin{defi} \label{d_focal}
Given a cone structure $(M,\C)$, a cone geodesic $\gamma:[a,b]\rightarrow M$ and a submanifold $P$ orthogonal to $\gamma$ in $\gamma(a)$,  an instant $s_0\in (a,b]$ is  {\em $P$-focal with multiplicity}  $r\in\N$, if it is $P$-focal with that multiplicity for one (and then all) Lorentz-Finsler metric compatible with $\C$.
\end{defi}

\subsection{Applications: Zermelo navigation problem and Wind Finsler} \label{s_63} 
Zermelo problem studies a (non-relativistic) object whose maximum speed at each point depends on both, the point and 
the (oriented) direction of its velocity. For such an object,   time-minimizing trajectories are searched. In classical Zermelo's, the variation of the velocity with the direction is determined by a vector field $W$ which represents the effect of a (time-independent) ``wind'' whose strength cannot be bigger than the maximum velocity developed by its engine. In this case,  it is known that such trajectories must be  geodesics for a certain Randers metric; this is determined by a ``background'' Riemannian one $g_0$ and the wind $W$, which must satisfy $g_0(W,W)<1$ (see  \cite{BCS04}).  Our aim here is to explain how cone geodesics permit to solve such a problem in a much more general setting,  including the possibility that the wind depends also on the time and it has arbitrary strength.  Obviously, this enlarges widely the applicability of the model (notice also that its applications include possibilities  far from the original one; see, for example, recent \cite{Markv16} about wildfire spread,  and the more recent one \cite{Gib17}). 

Let us start with the case when the object moves on a smooth manifold $S$ and  its  possible maximum velocities at each point $x$ depend on the direction. Then these maximum velocities are represented pointwise by the unit sphere for a Minkowski norm; globally, they determine a hypersurface  in $TS$ which corresponds with the indicatrix $\Sigma$ of a  (1-homogeneous)  Finsler metric $Z$, called the {\em Zermelo metric}. 
If $S$ is endowed with an auxiliary Riemannian metric $g_0$ with norm $|\cdot |_0$, then the maximum velocities can be represented by a positive function
$v_m: U_0S\rightarrow \R$, where $U_0$ is the $g_0$-unit bundle on $S$ and $v_m$ is determined by $v_m(u)u\in \Sigma$ for all $u\in U_0M$; as a consequence, $Z(u)=1/v_m(u)$.  Let $x,y\in S$ and let $\gamma:[a,b]\rightarrow S$ be a curve from $x$ to $y$ parametrized by $g_0$-length.
The time elapsed by  an object moving at maximum speed from $x$ to $y$ along $\gamma$ is given by:
\[T=\int_a^b \frac{ds}{v_m(\dot\alpha(s))}= 
\int_a^b Z(\dot\alpha(s)) ds.
\]
That is, the time is the length computed with the Finsler metric $Z$ (in particular, this length is independent of the reparameterization of $\alpha$, which can be dropped). Therefore, minimizing the time  for travelling from $x$ to $y$ is equivalent to finding a minimizing geodesic for $Z$. What is more, the functional {\em  arrival time  (at maximum speed) } or just $AT$ defined on the set of all the paths\footnote{This paths are {\em oriented}, that is, starting at $x$ and ending at $y$; recall that, in general $v_m(u)\neq v_m(-u)$ and, thus, $Z$ is not a reversible Finsler metric.} $\alpha$ from $x$ to $y$ becomes equal to the functional $Z$-length. So, the critical points for $AT$ are equal to the critical points for $Z$, i.e., the (minimizing or not) pregeodesics of $Z$. Now, consider the following extensions to these problems.

\subsubsection{Time-dependent Zermelo problem}
When the maximum speeds are time-dependent, then the natural setting of the problem is the following. Consider the product manifold $M=\R\times S$, where the natural projection $t: \R\times S\rightarrow \R$ represents the (non-relativistic) time and let $\Omega=dt$ and  $T=\partial_t$. 
Now, the $t$-dependent indicatrices provide a Finsler metric $Z$ on the bundle $\ker (\Omega)$ and, so, we have a  cone triple $
(\Omega =  dt,\partial_t,Z)$ and its corresponding cone structure $\C$.  Observe that at every point $(t,x)\in \R\times S$, $\ker(\Omega)$ can be identified with $T_xS$ 
 (i.e.
 $\ker(\Omega)_{(t,x)}\equiv \{t\}\times T_xS$) 
 and we can interpret $Z$ as a non-negative function $$Z:\R\times TS\rightarrow \R , $$ 
which is smooth   away from $\R\times \bf 0$ (being $\bf 0$ the zero section of  $ TS  $) 
 such that each $Z(t,\cdot)$,  $t\in\R$, is a Finsler metric  on $S$. 
Choosing  an instant of departure $t_0\in\R$, we must consider curves $\tilde \alpha$ 
which depart from $(t_0,x)$ and arrive at $\R\times \{y\}$, and look for first  arriving (or critical arriving) ones. 
With no loss of generality
we can assume that they are parametrized by $t$, i.e., $\tilde\alpha(t)=(t,\alpha(t))$, with $t\in [t_0,t_0+AT(\alpha)]$. Now, the restriction of travelling at a speed  no bigger  than  $v_m$ means that $\tilde \alpha$ is a {\em causal curve} for $\C$, and  $AT(\alpha)$ is again interpreted as the arrival time. The requirement of travelling  at maximum speed is equivalent to consider lightlike curves for $\C$. Observe that given a piecewise-smooth curve $\beta:[a,b]\rightarrow S$, there exists a unique (future-directed) lightlike curve $\tilde\beta:[a,b]\rightarrow \R\times S$, with $\tilde\beta(s)=(t(s),\beta(s))$.  Indeed $t(s)$ is obtained as a solution of $\dot t(s)=Z(t(s),\dot\beta(s))$ using \eqref{e_F} and, so the  arrival time  is computed as 
\begin{equation}\label{time-dependentcase}
 AT(\beta)= t(b)-t(a)  = \int_a^b Z(t(s),\dot\beta(s))ds
\end{equation}
 (recall that  the integral in \eqref{time-dependentcase} is independent of reparametrizations by the positive homogeneity of $Z$ in the second component).

 Now,  consider the set 
$\mathcal{P}_{((t_0,x),y)}$ of all the (piecewise smooth) lightlike curves from $(t_0,x)$ to $\R\times\{y\}$.
 The first arriving causal curve (if it exists) must be a lightlike curve and its arrival point will be  the first one  in the intersection of $J^+(t_0,x)\cap (\R\times\{y\})$. It is known in classical Causality Theory that, if this curve  exists, then it must be a lightlike pregeodesic. 
The extension of this result to the Finsler case  solves the time-dependent Zermelo problem  and it is a consequence of  Th.
\ref{t_CONE}.  

\begin{cor}\label{c_Zermelo}
Any local minimum of the $AT$ functional on $\mathcal{P}_{((t_0,x),y)}$  (i.e. any solution to the time-dependent
  Zermelo problem)  
 must be a cone geodesic of $\C$ without conjugate points  except, at most,  at  the endpoint.
Moreover:

\begin{itemize}
\item[(a)] A global minimum exists if the causal futures $J^+(t_0,x)$ in $\C$ are closed (i.e., the analogous property to  {\em
causal simplicity} of classical spacetimes  holds) 
and $\mathcal{P}_{((t_0,x),y)}\neq \emptyset$.

 Moreover,  the latter property (resp. the former one) is fulfilled  if $Z$ is upper bounded (resp. lower bounded) by any $t$-independent 
Finsler metric (resp. complete Finsler metric) on $\ker(dt)$. 

\item[(b)] All the  trajectories which are critical for the $AT$ functional must be cone geodesics.

\end{itemize}
\end{cor}

\begin{proof}
The first assertion holds because, if  the corresponding  minimum arrival point  $(T_,y)$ exists,  then the minimizer $\tilde \sigma$ must be a lightlike pregeodesic for any compatible Lorentz-Finsler metric with no conjugate points. 
Otherwise, by  Lemma  \ref{lem:closetimelike}, a connecting timelike curve would exist and, thus, a neighborhood of $(T_,y)$ could be joined by curves in $\mathcal{P}_{((t_0,x),y)}$.
Therefore, the result is a direct consequence of Th. \ref{t_CONE} and the discussion above Def. \ref{d_focal}. 

 For (a), the first assumption  ensures that $J^+(t_0,x)\cap (\R\times\{y\})$ is closed (with the component $\R$ lower bounded), and the second one that this intersection is not empty; so, the minimum for the reached $\R$-component yields the result. Notice also that the  upper boundedness of $Z$ implies that any curve in $S$ joining $x$ and $y$ can be lifted to a  causal  curve in $\R\times S$ from  $(t_0,x)$ to   $\R\times\{y\}$; so, $\mathcal{P}_{((t_0,x),y)}$ is not empty.  Moreover, the lower boundedness by some complete Finslerian metric $Z_0$ implies that  $J^+(t_0,x) \cap ([t_0,t_1]\times S)$ lies in a compact subset for any $t_1>t_0$. Then, for any converging sequence $\{(t_m,y_m)\}\rightarrow (t_1,y)$, with $ (t_m,y_m)\in J^+(t_0,x)$ for all $m$,  and  Zermelo curves $(t,\alpha_m(t))$, $t\in [t_0,t_m]$ from $(t_0,x)$ to $(t_m,y_m)$ the velocities $\dot\alpha_m$ are $Z_0$-bounded and Arzel\'a's theorem gives a Lipschitz limit curve $\tilde \alpha$ from $(t_0,x)$ to $(t_1,y)$. As in the standard Lorentzian case, the continuous  curve $\tilde \alpha$ is $\C$-causal in a natural sense (locally, its endpoints can be connected by a smooth causal curve) and, then, either $\tilde \alpha$ is a cone geodesic or  a timelike curve with the same endpoints as $\tilde \alpha$ exists\footnote{Indeed,  the lower boundedness by a complete $Z_0$  yields naturally the {\em global hyperbolicity} of any compatible $L$, being each slice $\{t\}\times S$ a Cauchy hypersurface (compare, for example, with the general result in \cite[Prop. 3.1]{San97})
 and, thus, the result is standard as in the Lorentzian case, where global hyperbolicity implies causal simplicity.}.

For (b), the Fermat relativistic principle
 developed by Perlick \cite{Perlick06} implies that the critical points for $AT$ correspond to the  lightlike pregeodesics for any Lorentz-Finsler metric compatible with $\C$ and, thus, they must be cone geodesics. 
\end{proof}
\begin{rem}
 Using orthogonality of cone geodesics to a submanifold, it is possible to solve the time-dependent  Zermelo problem in the case  of minimizing time when  one either departures from or arrives at  a smooth submanifold $P$. Indeed,  the solution in this case will be given by cone geodesics orthogonal to some $\{t_0\}\times P$  in an endpoint. 
\end{rem}
Obviously, the result   also holds in the time-independent case; the reader can check then that the continuous Finsler metric associated with $(dt,\partial_t,Z)$ is static and easily smoothable.

 Finally, it is also worth pointing out  the role of the maximum speeds $v_m$ in a relativistic setting.  The  setting of Zermelo's problem is  non-relativistic, however,  an obvious
relativistic interpretation arises
 once the cone structure is fixed and one thinks in $v_m$ as the maximum possible velocity measured by any observer  (at each event and direction)  for any particle, i.e., the (relativistic) {\em speed of light}. Under this viewpoint, the framework of the triple $(dt,\partial_t,F)$ (even in the case  $\R\times S \equiv \R^4$) can be useful to describe  either  possible anisotropies in the velocity of the light, or variations in its speed, a topic studied by quite a few authors \cite{Kos11,KosRus12,Rus15}.
 
\begin{rem}  Observe that our time-dependent metrics $Z:\R\times TS\rightarrow \R$ (for Zermelo problem) and the time computed with them (recall \eqref{time-dependentcase}) provide a rheonomic Lagrangian of Finsler type. Indeed, in \cite{Mark17}, S. Markvorsen studies the time-dependent Zermelo problem using rheonomic geometry as developed by M. Anastasiei et al. (see \cite[\S 7]{Anas94} and references therein,  and \cite{Bu10} for a different approach). As a consequence, a rheonomic Finsler-Lagrange geometry can be studied using lightlike geodesics of a Finsler spacetime.  In particular,  we can apply Cor. \ref{c_Zermelo} to obtain connecting results in a rheonomic Finsler-Lagrange geometry. 
\end{rem}
\subsubsection{Wind Finslerian structures}\label{s_632} As commented above, the metric  $Z$ in the classical Zermelo problem is a Randers metric for some pair $(g_0,W)$ with $g_0(W,W)<1$. However, Zermelo problem makes sense without this restriction. This general problem has been studied systematically in \cite{CJS14}, where the following results have been proved: 

(i) There exists a notion of {\em wind Riemannian structure}, which is a (seemingly singular) Randers-type metric where the pointwise $0$ vector do not lie inside the indicatrix. 

(ii) The geometry of such  structures, including their geodesics, is fully controlled by the cone structure $\C$ of  a  conformal class of  spacetimes, the {\em SSTK} ones. These  spacetimes (which are not by any means singular) admit a non-vanishing Killing vector field $K$ and, thus, they generalize  standard stationary spacetimes, where such a $K$ exists and  must also be   timelike.  

(iii) Zermelo problem can be described    
and solved by using the viewpoint of the SSTK spacetime. Moreover, the correspondence between both types of geometries yields quite a few interesting consequences including, for example,  a full understanding of the completeness of the Randers manifolds with constant flag curvature \cite{JS17} (classified in a celebrated paper by Bao, Robles and  Shen \cite{BCS04}). 

\smallskip

Given any Finsler metric and vector field $W$, one can consider analogously {\em wind Finslerian structures} (also defined in \cite{CJS14}), which is a much more general class of Finsler-type metrics generalizing  Randers-type ones. As in the case of wind Riemannian structures, wind Finslerian ones can also be  controlled by a cone structure $\C$. Moreover, such a $\C$ becomes invariant by the flow of  a non-vanishing vector field $K$, as in the case of SSTK spacetimes.  Again, there is a full correspondence  between the cone geodesics for $\C$ and the pregeodesics for the wind Finslerian structure.  Then, our study of cone geodesics here, including the compatibility with a Lorentz-Finsler metric $L$ (and the independence of the chosen $L$), becomes  sufficient elements to transplantate directly the results for wind Riemannian structures  in \cite{CJS14}  to the general wind Finslerian setting.

\appendix
  \section{Alternative definitions of Finsler spacetimes} \label{s_a1}
In the literature, there are several non-equivalent notions of Finsler spacetimes.  Next, we are going to  compare some of them which are related to ours. 
   Let us emphasize that there are quite a few of cosmological models using Lorentz-Finsler spacetimes from different viewpoints
  (see for example the review \cite{Vacaru} 
or  the more  recent \cite{PBPSS17})
which will not be considered specifically here. However,  as in the considered cases, our approach might be useful to understand their  global causal behavior. 
  \subsection{Beem's definition}  In this definition \cite{Beem70}, it is considered a pseudo-Finsler metric $L:TM\rightarrow \R$ with fundamental tensor having index $n-1$, where $n$ is the dimension of $M$. In this case, the restriction of $L$ to a connected  component  of $L^{-1}(-\infty,  0)  \subset TM\setminus \bf 0$  admitting a vector field $T$,   provides  a Finsler spacetime as introduced in \S \ref{s3.2}. In some cases, it is also required the metric $L$ to be reversible, namely, $L(-v)=L(v)$, for  every $v\in TM$. Otherwise,  whenever $L$ has two causal cones  (i.e., $L^{-1}(-\infty,  0)$ has two connected components as above),  one could choose such cones, as  the past and future cones; however, the absence of reversibility implies that the causal future and past would be unrelated. This makes   reasonable to focus  only on  one connected component of $L^{-1}(-\infty,0]\subset TM\setminus \bf 0$,  as in the present article. 
 Moreover,   
  taking into account the examples provided in \S \ref{s4}, which are not necessarily defined in the whole tangent bundle (when we consider a conic Finsler metric) or could be degenerate away from the causal cone,  the definition considered here focuses in the intrinsic properties of the, in principle, relevant part of the metric. Anyway, 
  it is interesting from the theoretical viewpoint that 
 any Lorentz-Finsler metric defined in a cone structure can be extended  to the whole tangent bundle as in Beem's definition, \cite{Min16} (not only the  connection as in Rem. \ref{globalCon}). 
 
 
 \subsection{Asanov's definition} 
 From  a more general viewpoint, this definition \cite{As85} does not consider as admissible those vectors in the boundary  of the causal cone. Namely, there is only a pseudo-Finsler metric $L:A\subset TM\rightarrow (0,+\infty)$ defined in a conic open subset which is convex and it has index $n-1$  (this possibility is also permitted in our definition of Lorentz-Minkowski norm when it is not proper, Def. \ref{d_LMnorm}, even though we have focused in the proper case). 
 Notice, however, that even if $L$ is not extendible to the boundary, the cone structure $\C$  obtained from the boundary of $A$  will make sense and, thus, our intrinsic study  of  $\C$  becomes  aplicable. 

 \subsection{Laemmerzahl-Perlick-Hasse's definition} In this definition \cite{LPH12}, the metric $L$ can be non-smooth in a set of measure zero, but the geodesic equation can be continuously extended to all the directions.  Remarkably, this definition includes simple  continuous Lorentz-Finsler metrics as the product of a Finsler metric by $(\R,-dt^2)$,  the metric  $G$ in \eqref{e_contLF} or the static ones in \cite{CaStan16}.  
In this framework, our study shows  (recall Rem. \ref{r_LAST}):  
(i)  there is a natural smooth cone structure $\C$ (according to our definition), whenever there exists an open  neighborhood  of the set of  all the lightlike directions where  $L$ is smooth and its indicatrix  is strongly convex 
\footnote{ This is not only applicable to \cite{LPH12}, but it also  recovers the definition of Finsler spacetime in \cite{AaJa16}}, 
  and (ii) $\C$ is not only compatible with some Lorentz-Finsler metric $L^*$ (according to our definition), but $L^*$ can  also be constructed as close to $L$ as desired  (by smoothing $L$ 
explicitly as in Th. \ref{t_principal}
and Lemma \ref{l_PRINC}). 
 Therefore, our study may clarify when the non-smoothability of $L$ is just a mathematical simplification of the model or when it is something inherent to it. Indeed, 
there  are other examples (as those in \cite{Min17}),  which are not smooth in the lightlike directions; thus,  even if geodesics can be defined there,  other fundamental quantities as flag curvature  (which would be  very relevant to study gravity)  are not available.
 \subsection{Kostelecky's definition} This definition, or better, examples, arises from effective models of the Standard-Model Extension \cite{Kos11} and it has a  strong  physical motivation and interest (see further developments in \cite{EdwKos18,KosRus12,Rus15}).
  The expression of the first examples of this kind presented in \cite{Kos11} is given by 
 \begin{equation*}
 L(v)=(\sqrt{-\tilde g(v,v)}+\tilde g(v,\av)+\varepsilon \sqrt{\tilde g(v,\bv)^2-\tilde g(\bv,\bv) \tilde g(v,v)})^2,
 \end{equation*}
 where $v$ belongs to the causal cone of  a classical Lorentzian metric $\tilde g$ (with index $1$), $\varepsilon^2=1$  and $\av$ and $\bv$ are two vector fields.   First of all, observe that if $\bv$ is timelike, then $\tilde g(v,\bv)^2-\tilde g(\bv,\bv) \tilde g(v,v)>0$ for all $v$ causal (Cauchy-Schwarz reverse inequality) and trivially, $\tilde g(v,\bv)^2-\tilde g(\bv,\bv) \tilde g(v,v)\geq 0$ if $\bv$ is non-timelike. Therefore, $L$ is well-defined in the subset of $\tilde g$-causal vectors. 
 Let us denote $F(v)=\sqrt{-\tilde g(v,v)}+\tilde g(v,\av)+\varepsilon\sqrt{p^\bv (v,v)}$,
where 
\[p^\bv(u,w)=\tilde g(u,\bv)\tilde g(w,\bv)-\tilde g(\bv,\bv) \tilde g(u,w),\]
for all $u,w\in TM$.
 We will study $F$ as a Randers-type modification of the case $\av=0$, so let   $\hat F(v)=\sqrt{-\tilde g(v,v)}+\varepsilon\sqrt{p^{\bv}(v,v)}$ (which can be non-positive).  In  order to check whether the fundamental tensor has index $n-1$,  Prop. \ref{indexn-1} will be applied. 
 From \eqref{eq:angularm}, the angular metric $\hat h_v$ (of $\hat L=\hat F^2$) for a $\tilde g$-timelike vector $v$ is given by
 \begin{multline*}
 \hat h_v(u,u)=-\frac{\hat F(v)}{\sqrt{-\tilde g(v,v)}} \tilde h_v(u,u)
 	+\varepsilon\frac{\hat F(v)}{\sqrt{p^\bv(v,v)}}\left(p^\bv(u,u)-\frac{p^\bv(v,u)^2}{p^\bv(v,v)}\right),
 \end{multline*}
 where $\tilde h_v$ is the angular metric of $\tilde g$.
 Observe that if $u\in \langle v\rangle^{\perp_{\tilde g}}\cap  \langle \bv\rangle^{\perp_{\tilde g}}$, then
 \[\hat h_v(u,u)=-\hat F(v)\left(\frac{1}{\sqrt{-\tilde g(v,v)}}+\varepsilon \frac{\tilde g(\bv,\bv)}{\sqrt{p^\bv(v,v)}} \right)\tilde g(u,u),\]
$\hat h_v(\bv,u)=0$ and $\hat h_v(\bv,\bv)=\frac{\hat F(v)}{\tilde g(v,v)\sqrt{-\tilde g(v,v)}}p^\bv(v,v)$. Putting all this together and recalling that $v$ is in the radical of $\hat h_v$, it follows that $\hat h_v$ is negative semi-definite with radical generated by $v$ if and only if 
 \begin{equation*}
 \hat F(v)>0\quad \text{and} \quad \frac{1}{\sqrt{-\tilde g(v,v)}}+\varepsilon \frac{\tilde g(\bv,\bv)}{\sqrt{p^\bv(v,v)}} >0,
 \end{equation*}
 which is equivalent to 
\begin{equation}\label{eq:kost}
-\tilde g(v,v)+\varepsilon p^\bv (v,v)>0\,\, \text{and}\,
\begin{cases}
\text{either}\quad \varepsilon \tilde g(\bv,\bv)\geq 0\\
\text{or}\, \varepsilon \tilde g(\bv,\bv)< 0 \text{ and } p^\bv(v,v)+\tilde g(v,v)\tilde g(\bv,\bv)^2>0.
\end{cases}
\end{equation}
Moreover,  $\hat h_v$ is positive semi-definite with radical generated by $v$ if and only if 
\begin{equation*}
\hat F(v)<0\quad \text{and} \quad \frac{1}{\sqrt{-\tilde g(v,v)}}+\varepsilon \frac{\tilde g(\bv,\bv)}{\sqrt{p^\bv(v,v)}} <0,
\end{equation*}
which is equivalent to
\begin{equation}\label{eq:kost2}
\varepsilon=-1,\,-\tilde g(v,v)-p^\bv (v,v)<0,\, \tilde g(\bv,\bv)> 0 \text{  and   } p^\bv(v,v)+\tilde g(v,v)\tilde g(\bv,\bv)^2<0.
\end{equation}
In the other cases, $\hat h_v$ is indefinite. This implies that when $\av=0$, the fundamental tensor of $L$ has index $n-1$ if and only if \eqref{eq:kost} holds (recall Prop. \ref{indexn-1}).
    Let us study the general case with $\av\not=0$. By a direct computation, (see also  \cite[Cor. 4.17]{JavSan11}):
 \begin{equation*}
 g_v(w,w)=\frac{F(v)}{\hat F(v)}\hat h_v(w,w)+\left(\frac{\hat g_v(v,w)}{\hat F(v)}+\tilde g(w,\av)\right)^2.
 \end{equation*}
  Then its angular metric is given by 
 \begin{equation*}
 h_v(w,w)= \frac{F(v)}{\hat F(v)}\hat h_v(w,w).
 \end{equation*}
  By Prop. \ref{indexn-1},  $g_v$ has index $n-1$ if 
   $F(v)>0$ and either  \eqref{eq:kost} 
 or \eqref{eq:kost2} hold, and trivially, $g_v$ is degenerate if $F(v)=0$. Moreover, $g_v$ is not defined when $\tilde g(v,v)=0$, as $L$ is not smooth there.  Therefore, considering a connected region such that $F(v)\geq 0$,  the fundamental tensor $g_v$ has index $n-1$ for $v$ in the interior satisfying \eqref{eq:kost} or \eqref{eq:kost2}, and it is degenerate for $v$  in the boundary of the region;   what is more,  the region $\{v\in TM: \text{$F(v)>0$ and $\tilde g(v,v)<0$}\}$ is not empty if and only if there exists a $p\in M$ such that one of the intersections $\omega_\av^{-1}(-1)\cap  \hat {\mathfrak{B}}^+_p$ or $\omega_\av^{-1}(1)\cap  \hat {\mathfrak{B}}^+_p$ are non empty, where 
 $\omega_\av(v)=\tilde g(v,\av)$ for $v\in TM$ and 
 \begin{align*}
 \hat {\mathfrak{B}}^+_p&=\{v\in T_pM:\hat F(v)\geq  1, \text{$v$ satisfies \eqref{eq:kost}}\},\\
  \hat {\mathfrak{B}}^-_p&=\{v\in T_pM:\hat F(v)\leq  -1, \text{$v$ satisfies \eqref{eq:kost2}}\}
  \end{align*}
   (compare with Th. \ref{t_examp} and Cor.~\ref{c_examp}). 
 
  Observe that there are  three  possibilities at every point $p\in M$:
 \begin{enumerate}
 \item[(i)] The intersection  $\omega_\av^{-1}(\mp 1)\cap  \hat {\mathfrak{B}}^\pm_p$ has  empty interior, and then $L$ is not of Lorentz type at any point in the causal cone of $\tilde g$.
 \item[(ii)]  The intersection  $\omega_\av^{-1}(\mp 1)\cap  \hat {\mathfrak{B}}^\pm_p$ is  compact and (i) does not hold.  Then $L$ determines a cone structure if $\omega_\av^{-1}(\mp 1)\cap  \hat {\mathfrak{B}}^\pm_p$ does not touch the lightlike cone of $\tilde g$ and it has index $n-1$ when $F>0$, but not in the boundary.  This means that the causality of such metrics can be studied with our approach. 
 \item[(iii)] The intersection  $\omega_\av^{-1}(\mp 1)\cap  \hat {\mathfrak{B}}^\pm_p$ is non-compact and (i) does not hold. The region where $F>0$ does not determine a cone structure (the boundary is not smooth), but the fundamental tensor of $L$ has index $n-1$ there.
 \end{enumerate}

  \subsection{Pfeifer-Wohlfarth's definition} The main feature of this definition \cite{PW11} is that the Lorentz-Finsler $L$ does not necessarily extend smoothly to  the lightcone  (or the extension has not Lorentzian index there) but  such a property holds  for some power of $L$. Thus,  the Finslerian connections determined by $L$ are not defined in the lightlike directions; however, these authors show that it is possible to extend the connection defined in the timelike directions to the lightlike cone.  Our viewpoint on cone structures is applicable here
  as it can be proved that the Pfeifer-Wohlfarth's lightlike cone is a cone structure in our sense (proceed  as in part $(iv)$ of Prop.~\ref{propiedades}  with a $p$-homogeneous Lagrangian  which has  Lorentzian fundamental tensor).  
  
   Observe that one of the main examples of Pfeifer-Wohlfarth's definition,  namely, the bimetrics, can be generalized using the examples given in \S \ref{s4}. More precisely, if $L_1:A_1\rightarrow[0,+\infty)$ and $L_2:A_2\rightarrow [0,+\infty)$ are two Lorentz-Finsler metrics (according to our definition), then $L:A_1\cap A_2\rightarrow [0,+\infty)$ defined as $L(v)=\sqrt{L_1(v)L_2(v)}$ has fundamental tensor with Lorentzian index on the whole $A_1\cap A_2$ (with independence of its behaviour in the boundary). Indeed, its fundamental tensor is given by 
\begin{multline*}
g_v(u,w)=\frac{1}{2L(v)}(p_v(u,w)-\frac{1}{L(v)^2}p_v(v,u)p_v(v,w)) \\
+\frac{1}{L(v)}(g_v^1(v,u)g_v^2(v,w)+g_v^1(v,w)g_v^2(v,u)) 
\end{multline*}
where $g_v^1$ and $g_v^2$ are, respectively, the fundamental tensors of $L_1$ and $L_2$ and now, $p_v(u,w)=g_v^1(u,w)L_2(v)+g_v^2(u,w)L_1(v)$ for $u,w\in T_{\pi(v)}M$ and $v\in A_1\cap A_2$. Moreover, the angular metric of $L$ is given by 
\begin{multline*}
h_v(u,u)=\frac{1}{2L(v)}(p_v(u,u)-\frac{3}{2 L(v)^2}p_v(v,u)^2)+\frac{2}{L(v)}g_v^1(v,u)g_v^2(v,u).
\end{multline*}
Observe that if $u\in \langle v\rangle^{\perp g^1_v}$, then
\begin{align*}
h_v(u,u)&=\frac{1}{2L(v)}(L_1(v) g_v^2(u,u)+L_2(v) g^1_v(u,u)-\frac{3L_1(v)}{L_2(v)} g_v^2(v,u)^{2})\\
&= \frac{1}{2L(v)}( L_1(v) h_v^2(u,u)+L_2(v) h_v^1(u,u)-\frac{L_1(v)}{L_2(v)} g_v^2(v,u)^{2}),
\end{align*}
where $h_v^1$ and $h_v^2$ are the angular metrics of $L_1$ and $L_2$. Applying Prop.~\ref{indexn-1} to $L_1$ and $L_2$, we deduce that $h_v$ is negative semi-definite with radical generated by $v$, which, again by Prop. \ref{indexn-1}, implies that $g_v$ has index $n-1$ as required. If $\bar A_1\subset A_2$, then it is possible to show with similar techniques that the Hessian of $L^2$ is of Lorentzian type in the boundary of $A_1$ (observe that $L$ is not smooth there); so, this case lies under Pfeifer-Wohlfarth's definiton and yields a cone structure according to our definition, as claimed above. Finally, if $A_1=A_2$, then $L$ is a  Lorentz-Finsler metric according to our definition, because if $v$ is in the boundary of $A_1=A_2$, then
\[g_v(v,u)=\frac{1}{2}\left(\sqrt{\frac{L_2(v)}{L_1(v)}}g_v^1(v,u)+\sqrt{\frac{L_1(v)}{L_2(v)}}g_v^2(v,u)\right),\]
and then Th. \ref{t_anisotropic} and Prop. \ref{indexn-1} conclude.

A further example that can be generalized is the one provided by Bogoslovsky (see \cite{Bogos07} and references therein), which turned out to be a model for very special relativity \cite{GGP07} and it has been considered recently to model pp-waves \cite{FP16}. As it was observed in \cite{FPP18}, the Bogoslovsky metric  satisfies the conditions in Pfeifer-Wohlfarth's definition. Consider now the Lagrangian $L(v)=L_0(v)^{(1+b)}/\beta(v)^{2b}$, where $L_0:A\rightarrow [0,+\infty)$ is a Lorentz-Finsler metric and $\beta$ a one-form such that $\bar A\cap \ker \beta=\bf 0$. Then if $F_0=\sqrt{L_0}$, following \cite[Cor. 4.12]{JavSan11}, we obtain the fundamental tensor of $L:A\rightarrow [0,+\infty)$ as
\begin{multline*}
\frac{L_0(v)}{L(v)}g_v(u,u)=(b+1) h_v^0(u,u)\\
+b(b+1)\left(\frac{g^0_v(v,u)}{F_0(v)}-\frac{F_0(v)}{\beta(v)}\beta(u)\right)^2\\+\left((b+1)\frac{g^0_v(v,u)}{F_0(v)}-b\frac{F_0(v)}{\beta(v)}\beta(u)\right)^2,
\end{multline*}
where $g_v^0$ is the fundamental tensor of $L_0$ and the angular metric of $L$  is 
\[\frac{L_0(v)}{L(v)}h_v(u,u)=(b+1)h^0_v(u,u)+b(b+1)\left(\frac{g^0_v(v,u)}{F_0(v)}-\frac{F_0(v)}{\beta(v)}\beta(u)\right)^2,\]
which is negative semi-definite with radical generated by $v$, whenever $-1<b<0$. The Lagrangian $L$ is not necessarily smooth in the boundary of $A$, but so is the power $L^{\frac{1}{1+b}}$. It is possible to show that the Hessian of this power is of Lorentzian type in the lightcone of $L$ and then $L$ defines a Pfeifer-Wohlfarth's Finsler spacetime.

\section{Lorentz-Minkowski norms}\label{s_a2}

For  a classical norm on a vector space $V$ (eventually conic or non-reversible)  there is a well-known relation between the convexity of its indicatrix, the triangle inequality and the 
fundamental inequality  for its fundamental tensor (see \cite{JavSan11} for a  summary). 
These relations are easily transplanted to Lorentz-Minkowski norms, namely: 

\begin{prop}\label{LorentzTriangle}
Let  $A\subset V$ be a conic salient domain, $F:A\rightarrow  \R$, $F>0$, a continuous   two-homogeneous positive function,   $\Sigma=F^{-1}(1)$ be its indicatrix,   $B:=F^{-1}([1,+\infty))$ and, in the case that $F$ is smooth, with  fundamental tensor $g$ as in \eqref{fundtens}. Then:
\begin{enumerate}
\item $F$ satisfies the reverse triangle inequality 
\begin{equation}\label{EqRevTI}
F(v+w)\geq F(v)+F(w)\qquad\forall v,w\in A
\end{equation}
if and only if $B$ is convex. 

When $F$ is smooth, this is equivalent to (i) the positive semi-definiteness of the second fundamental form $\sigma^\xi$ of $\Sigma$ with respect to the position vector $\xi$ and (ii) the negative semi-definitess of $g$ on $\Sigma$.   
Moreover, in this case $g$ satisfies the  non-strict, reverse fundamental inequality, 
\begin{equation}\label{RCS}
g_v(v,w)\geq F(v)F(w) \qquad \forall v,w\in A.
\end{equation}

\item $F$ satisfies the strict reverse triangle inequality (i.e., \eqref{EqRevTI} holds with equality   only when $v, w$ are collinear) if and only if $B$ is strictly convex (i.e., each open segment with endpoints $v,w\in B$ is included in the interior of $B$ except at most $v,w$). 

When $F$ is smooth, this is equivalent to the strict convexity of $\Sigma$ with respect to $B$ (i.e. the hyperplane tangent to $\Sigma$ at any point only touches $B$ at that point). 

This property holds when $g$ is non-degenerate with index $n-1$  (in particular, when $F$ is a Lorentz-Minkowski norm). In that case, the {\em (strict) reverse fundamental inequality} holds (i.e., \eqref{RCS} holds with equality  if and only if $v,w$ are collinear); such an inequality
 becomes the classical {\em reverse Cauchy-Schwarz} one when $F$ comes from a Lorentzian scalar product.

\end{enumerate}
\end{prop}
Some proofs of these assertions are spread in the literature (see for example \cite{AaJa16,Min15}) and a detailed development is carried out in \cite{Morales}; 
the latter is  sketched here for the convenience of the reader.

\smallskip

\begin{proof}[Proof of Prop. \ref{LorentzTriangle}]
All the assertions  follow by using  the Euclidean arguments in \cite[Prop. 2.3]{JavSan11} (see \cite[Sect. 3.2.1]{Morales} for details), except those involving the fundamental inequality \eqref{RCS}. The latter is equivalent to
\begin{equation}\label{RCSbis}
\df F_v(w) \geq F(w) \qquad \forall v,w\in A,
\end{equation}
and consider the non-trivial case when $v$, $w$ are not collinear.  Then if $\tilde h:=\textrm{Hess}(F)$, using that $F=\sqrt{L}$, it follows straightforwardly that 
\[\tilde h_v(u,w)=\frac{1}{F(v)^3} h_v(u,w),\]
for $v\in A$ and $u,w\in V$, where $h_v$ is the angular metric of $L$ (recall \eqref{angularmetric}). Therefore, recalling Prop. \ref{indexn-1}  
%
\begin{equation}\label{NegatDefinith}
h_v(u,u)\leq 0 \qquad \forall v\in A, \quad \forall u\in V. 
\end{equation}
 (where  the radical of $h_v$ may contain directions different to $v$  only if  
 $g$ is degenerate). 
If $v-u \in A$, the second mean-value theorem at $0$ yields 
\begin{equation*}
F(v-u)=F(v)-\df F_v(u)+\frac{1}{2}h_{v+\delta u}(u,u) \quad\textrm{for some $\delta\in (-1,0)$},
\end{equation*}
and using \eqref{NegatDefinith}, 
\begin{equation}\label{AuxiliaryIneq2}
F(v-u)\leq F(v)-\df F_v(u).
\end{equation}
So, \eqref{RCSbis} follows by putting $u:=v-w$ and recalling $\df F_v(v)=F(v)$. 
\end{proof}

\begin{rem}  It is worth pointing out the following relation between the Lorentz-Finsler and classical Finsler cases. Let $A$ be as above,  $L: A\rightarrow \R, L>0$
smooth and {\em $r$-homogeneous} ($L(\lambda v)=\lambda^{r}L(v)$ for some $r\neq 0$), and $\Sigma_a=L^{-1}(a)$, $a>0$. Then  Hess$(L)$ is negative definite (resp. 
  semi-definite) 
on one (and then all) $\Sigma_a$  if and only if
$\hess (1/L)$ is positive  definite (resp.  semi-definite) there\footnote{Now $\Sigma_a$ satisfies
$g_v=-a\cdot r\cdot \sigma_v^\xi$, 
as $g_v(u,v)=(r-1) dL_v(u),  g_v(v,v)=r(r-1)L(v)$.
}. As a consequence, when $r>1$:

 $\hess L$ {\em  has Lorentzian 
signature $(+,-,\dots ,-)$ (resp. coindex one)  if and only if} $\hess(1/L)$ {\em is positive definite (resp. semi-definite).}

Notice also that $\hess F$ can be written in terms of $\hess (1/F^2)$, which can be used alternatively to prove the reverse triangle inequality \cite[Prop. 8.7]{Morales}.
\end{rem}

\section*{Acknowledgments}
 The authors warmly acknowledge Professor Daniel Azagra (Universidad Complutense, Madrid) his advise on approximation of convex  functions as well as Profs. Kostelecky (Indiana University), Fuster (University of Technology, Eindhoven), Stavrinos (University of Athens), Pfeifer (University of Tartu), Perlick (University of Bremen) and  Makhmali (Institute of Mathematics, Warsaw) their comments on a preliminary version of the article.  The careful revision by the referee is also acknowledged. 
 
 This work is a result of the activity developed within the framework of the Programme in Support of Excellence Groups of the Regi\'on de Murcia, Spain, by Fundaci\'on S\'eneca, Science and Technology Agency of the Regi\'on de Murcia. MAJ  was partially supported by MINECO/FEDER project reference MTM2015-65430-P and Fundaci\'on S\'eneca project reference 19901/GERM/15, Spain and  MS by 
Spanish  MINECO/ERDF project reference MTM2016-78807-C2-1-P. 



\end{document}